\documentclass[9pt]{article}

\usepackage{graphicx,subfig}
\graphicspath{{./pics/}}
\usepackage{bm,color}
\usepackage{amssymb,amsmath,amsthm,amsfonts}
\usepackage[section]{algorithm}
\usepackage{algorithmic}
\usepackage{enumerate}
\usepackage{verbatim}
\usepackage[colorlinks,linktocpage,linkcolor=blue]{hyperref}
\numberwithin{equation}{section}

\theoremstyle{plain}
\newtheorem{theorem}{Theorem}[section]
\newtheorem{lemma}[theorem]{Lemma}
\newtheorem{proposition}[theorem]{Proposition}

\newtheorem{assumption}{Assumption}[section]
\theoremstyle{theorem}

\newtheorem{remark}{Remark}[section]

\def\beq#1\eeq{\begin{equation}#1\end{equation}}
\def\bal#1\eal{\begin{aligned}#1\end{aligned}}
\def\tu{\tilde u}

\def\al{{\alpha}}

\def\RR{{\mathbb R}}
\def\II{{\mathbb (D)}}
\def\CC{{\mathbb C}}

\newcount\icount

\def\DD#1#2{\icount=#1
  \ifnum\icount<1
  \,_{ 0}\kern -.1em D^{#2}_{\kern -.1em x}
  \else
  \,_{x}\kern -.2em D^{#2}_1
  \fi
}

\def\DDRI#1#2{{\icount=#1
  \ifnum\icount<1
  \,_{-\infty}^{\kern 1em R}\kern -.2em D^{#2}_{\kern -.1em x}
  \else
  \,_{x}^R \kern -.2em D^{#2}_\infty
  \fi
}}

\def\DDR#1#2{{\icount=#1
  \ifnum\icount<1
 _{0}^{ \kern -.1em R} \kern -.2em D^{#2}_{\kern -.1em x}
  \else
 _{x}^{ \kern -.1em R} \kern -.2em D^{#2}_{\kern -.1em 1}
  \fi
}}

\def\DDCI#1#2{{\icount=#1
  \ifnum\icount<1
  \,_{-\infty}^{\kern 1em C}  \kern -.2em D^{#2}_{\kern -.1em x}
  \else
  \,_{x}^C \kern -.2em  D^{#2}_\infty
  \fi
}}

\def\DDC#1#2{{\icount=#1
  \ifnum\icount<1
  \,_{0}^C \kern -.2em  D^{#2}_{\kern -.1em x}
  \else
  \,_{x}^C \kern -.2em D^{#2}_1
  \fi
}}

\def\Hd#1{\widetilde H^{#1}}   
\def\Hdi#1#2{\icount=#1
  \ifnum\icount<1
  \widetilde H_{L}^{#2}\II
  \else
  \widetilde H_{R}^{#2}\II
  \fi
}

\def\Cd#1{\icount=#1
  \ifnum\icount<1
  \widetilde C_{L}
  \else
  \widetilde C_{R}
  \fi
}

\title{A Finite Element Method for the Fractional Sturm-Liouville Problem}
\author {Bangti Jin\thanks{Department of Mathematics and Institute for
Applied Mathematics and Computational Science, Texas A\&M University,
College Station, TX 77843-3368 ({\texttt{btjin,lazarov,pasciak,rundell@math.tamu.edu}})}
\and Raytcho Lazarov\footnotemark[1] \and Joseph Pasciak\footnotemark[1] \and William Rundell\footnotemark[1]}

\begin{document}

\maketitle


\begin{abstract}
In this work, we propose an efficient finite element method for solving fractional Sturm-Liouville problems
involving either the Caputo or Riemann-Liouville derivative of order $\alpha\in(1,2)$ on the unit interval
$(0,1)$. It is based on novel variational formulations of the eigenvalue problem.
Error estimates are provided for the finite element approximations of the eigenvalues. Numerical
results are presented to illustrate the efficiency and accuracy of the method. The results indicate
that the method can achieve a second-order convergence for both fractional derivatives, and can provide accurate approximations to
multiple eigenvalues simultaneously.
\end{abstract}

\section{Introduction}\label{sec:introduction}
We consider the following fractional Sturm-Liouville problem (FSLP): find $u$ and $ \lambda \in \CC$ such that
\begin{equation}\label{eigen-strong}
  \begin{aligned}
   -D_0^\alpha u(x) + qu = \lambda u, & \quad x \in D=(0,1),\\
   u(0)=u(1)=0, &
   \end{aligned}
\end{equation}
where the fractional order $\alpha\in(1,2)$ and $D_0^\al$ refers to the left-sided Caputo or Riemann-Liouville
fractional derivative defined below by \eqref{eqn:Caputo} and \eqref{eqn:Riemann}, respectively. The potential
coefficient $q(x)$ is a (not necessarily nonnegative) measurable function in $ L^\infty\II$.
When $\alpha$ equals $2$, either fractional derivative coincides with
 the usual second-order derivative $u''$
\cite[eq. (2.1.7) and eq. (2.4.55)]{KilbasSrivastavaTrujillo:2006}.

The interest in the FSLP \eqref{eigen-strong} has several  motivations.
The first arises in the study of materials with memory, where the fractional-order derivative in
space is often used to describe anomalous (super-) diffusion processes (see the comprehensive
review \cite{MetzlerKlafter:2000}). The second motivation comes from  the analysis of (space) fractional
wave equation or fractional Fokker-Plank equation. 
In this case the space fractional derivative admits statistical interpretation as the macroscopic
counterpart of Levy motion (as opposed to Brownian motion for standard diffusion process);
see \cite{BensonWheatcraftMeerschaert:2000} for the derivation for solute transport in subsurface
materials. Third, it underlies M. Djrbashian's construction on spaces of analytical functions
\cite[Chapter 11]{Djrbashian:1993}, where the eigenfunctions of problem similar to \eqref{eigen-strong} are
used to construct a certain bi-orthogonal basis. This construction dates at least back to \cite{Dzrbasjan:1970}; see also \cite{Nahusev:1977}.
Mathematically, it also serves as a natural departure point from the classical Sturm-Liouville
problem, for which there is a wealth of deep mathematical results and efficient numerical methods
\cite{ChadanColton:1997}.

The FSLP \eqref{eigen-strong} is closely connected with the two-parameter Mittag-Leffler
function \cite{Podlubny_book}
\begin{equation*}
E_{\alpha,\beta}(z)=\sum_{k=0}^\infty\frac{z^k}{\Gamma(k\alpha+\beta)},\quad z\in\mathbb{C}.
\end{equation*}
It is well known that problem \eqref{eigen-strong} for $q(x)\equiv0$ has infinitely many eigenvalues
$\lambda$ that are the zeros of the Mittag-Lefler
functions $E_{\al,2}(-\lambda)$ and $E_{\alpha,\alpha}(-\lambda)$ for the Caputo and Riemann-Liouville
fractional derivatives, respectively (see \cite{Dzrbasjan:1970,Nahusev:1977} for related discussions),
and the eigenfunctions can also be expressed in terms of Mittag-Leffler functions.
Further, it is known that only a finite number of eigenvalues are real and all other eigenvalues
are genuinely complex. The asymptotic distribution of the eigenvalues is also known \cite{Sedletskii:2000,
JinRundell:2012}. However, computing the zeros of the Mittag-Lefler function
in a stable and accurate way remains a very challenging task (in fact, evaluating the Mittag-Lefler function
$E_{\alpha,\beta}(z)$ to a high accuracy is already highly nontrivial \cite{Seybold:2008}).

Al-Mdallal \cite{AlMdallal:2009} presented a numerical scheme for the FSLP with a Riemann-Liouville
derivative based on the Adomian decomposition method.  However, there is no mathematical analysis, e.g.,
convergence and error estimates, of the numerical scheme, and it can only locate one eigenvalue each
time. In \cite{JinRundell:2012}, by reformulating the FSLP with a Caputo derivative as an initial value
problem, a Newton type method was proposed for computing eigenvalues. However, since it generally
involves many solves of (possibly very stiff) fractional ordinary differential equations, the method
is expensive. Hence, there seems no fast, accurate and yet justified algorithm for the FSLP
available in the literature. In this work, we present and analyze a finite element method (FEM) for
computing the eigenvalues and eigenfunctions in problem \eqref{eigen-strong}. It is especially suited
to the case with a general potential $q$, and can provide accurate approximations to multiple eigenvalues
simultaneously with provable error estimates. We believe that the scheme provides a valuable tool for
studying refined analytical properties, e.g., asymptotics, bifurcation, and interlacing property, of
the FSLP \eqref{eigen-strong}, such as those listed in \cite{JinRundell:2012} for the Caputo derivative.

In order to apply the FEM, we need the weak formulation of problem \eqref{eigen-strong}. This itself is not
a straightforward task and has been addressed in a reasonable way only very recently \cite{ErvinRoop:2006,
JinLazarovPasciak:2013a}. We introduce the bilinear form $a(u,v): U \times V \to \CC$ for suitable solution
and test spaces $U$ and $V$ in Section \ref{sec:prelim}. In either case, the bilinear form $a(\cdot,\cdot)$
is nonsymmetric. While for the Riemann-Liouville case we can take $U=V=\Hd {\alpha/2}\II$ (see Section
\ref{ssec:def} below for the definition), the spaces are different for the Caputo case. Then the weak formulation
of the eigenvalue problem \eqref{eigen-strong} reads: find $u \in U$ and $\lambda \in \CC$ such that
\begin{equation}\label{eigen-weak}
a(u,v)=\lambda (u,v) \mbox{    for all } v \in V.
\end{equation}
The finite element approximation of problem \eqref{eigen-weak} is introduced in Section \ref{sec:fem}.
Specifically, we define finite dimensional subspaces $U_h \subset U$ and $V_h \subset V$,
and seek an approximation $u_h \in U_h$ and $\lambda_h \in \CC$ such that
\begin{equation}\label{eigen-weak-h}
a(u_h,v)= \lambda_h (u_h,v) \mbox{    for all  } v \in V_h.
\end{equation}
The main theoretical result of the paper is stated in Theorem \ref{eigencval-error}, i.e., $| \lambda
- \lambda_h | \le Ch^r$, where $h$ denotes the mesh size. Provided that the eigenvalue $\lambda$ is
simple, the exponent $r$ is given by $r<\al-1$ for the Riemann-Liouville case, while for the Caputo
case if $q\in \Hd s\II\cap L^\infty\II$, $s\in[0,1]$ and $\alpha+s>3/2$, then $r<\min(\alpha+s,2)-1/2$.

The rest of the paper is organized as follows. In Section \ref{sec:prelim}, we describe the variational
formulations of the source problem, and recall relevant regularity results on the variational solutions.
Then in Section \ref{sec:fem}, we introduce the finite element method. We shall discuss details for its
efficient implementation and provide rigorous error bounds. Readers who are not interested in the
technical derivations may simply skip Sections \ref{ssec:reg} and \ref{ssec:FEM-eigen}. In Section
\ref{sec:numerics}, we present extensive numerical experiments to illustrate the convergence behavior
and efficiency of the method, and briefly discuss possible extensions and its application in the study
of fractional Sturm-Liouville problems.

\section{Variational formulations for source problem}\label{sec:prelim}

To derive the finite element method, we need the variational formulations of the
following source problem
\begin{equation}\label{strongp}
  \begin{aligned}
    -D_0^\alpha u(x) + qu & = f, \quad x\in (0,1),\\
     u(0)&=u(1)=0,
  \end{aligned}
\end{equation}
where $f\in L^2\II$ or suitable Sobolev space. We only give an informal derivation here, and refer interested
readers to \cite{JinLazarovPasciak:2013a} for rigorous justifications. Also we recall the smoothing properties
of the source problem, which will be essential for the error analysis in Section \ref{sec:fem}.

\subsection{Notation for functional spaces}\label{ssec:def}
We first introduce notation for fractional order Sobolev spaces. Since the spectrum of the fractional
differential operator $-D_0^\alpha+q$ lies in the complex plane $\CC$, we need to consider complex-valued
functions. The norms in $L^2\II$ and $H^1\II$ are defined through the inner product $(u,v) = \int_0^1 u(x)
\bar v(x) dx$, where $\bar{v}(x)$ refers to the complex conjugate of $v(x)$ and $(u,v)_{H^1\II} =  ( u,v)
+ (u^\prime, v^\prime)$. The fractional order Sobolev spaces $H^\beta\II$, $\beta$ non-integer, are defined
through the real method of interpolation, and the norm is denoted by $\|\cdot\|_{H^\beta\II}$. For any
$\beta\ge 0$, we denote $\Hd \beta\II$ to be the set of functions in $H^\beta\II$ whose extension by zero
to $\RR$ are in $H^\beta(\RR)$ \cite{AdamsFournier:2003}, and $\Hdi 0 \beta$ to be the set of functions $u$
on $D$ whose extension by zero $\tilde{u}$ are in $H^\beta(-\infty,1)$. For $u\in \Hdi 0 \beta$, we set
$\|u\|_{\Hdi 0\beta}:=\|\tu\|_{H^\beta(-\infty,1)}$. Let $\Cd0^\beta\II$ denote the set of functions in
$v\in C^\infty[0,1]$ satisfying $v(0)=v^\prime(0)=\ldots=v^{(k)}(0)=0$ for $k\le\beta-1/2$. Throughout
we use $c$ to denote a generic constant, which may change at different occurrences, but it is always
independent of the mesh size $h$, and $c_\alpha$ to denote a nonzero constant only depending on $\alpha$.

\subsection{Fractional derivatives and integrals}
We first recall the definition of Caputo and  Riemann-Liouville fractional derivatives. For any positive
non-integer  real number $\beta$ with $n-1 < \beta < n$, the (formal) left-sided Caputo fractional derivative
of order $\beta$ is defined by (see, e.g., \cite[pp. 92]{KilbasSrivastavaTrujillo:2006}, \cite{Podlubny_book})
\begin{equation}\label{eqn:Caputo}
  \DDC 0 \beta u = {_0\hspace{-0.3mm}I^{n-\beta}_x}\bigg(\frac {d^nu} {d
    x^n }\bigg)
\end{equation}
and the (formal) left-sided Riemann-Liouville fractional derivative of order $\beta$
is defined by \cite[pp. 70]{KilbasSrivastavaTrujillo:2006}:
\begin{equation}\label{eqn:Riemann}
  \DDR0\beta u =\frac {d^n} {d x^n} \bigg({_0\hspace{-0.3mm}I^{n-\beta}_x} u\bigg) .
\end{equation}
Here $_0\hspace{-0.3mm}I^{\gamma}_x$ for $\gamma>0$
is the left-sided Riemann-Liouville integral operator of order $\gamma$ defined by
\begin{equation*}
 ({\,_0\hspace{-0.3mm}I^\gamma_x} f) (x)= \frac 1{\Gamma(\gamma)} \int_0^x (x-t)^{\gamma-1} f(t)dt,
\end{equation*}
where $\Gamma(\cdot)$ is the Euler's Gamma function defined by $\Gamma(x)=\int_0^\infty t^{x-1}e^{-t}dt$. As
the order $\gamma$ approaches 0, we can identify the operator $_0\hspace{-0.3mm}I_x^\gamma$ with the identity
operator \cite[pp. 65, eq. (2.89)]{Podlubny_book}. The integral operator ${_0\hspace{-0.3mm}I^\gamma_x}$
satisfies a semigroup property, i.e., for $\gamma,\delta> 0$ and smooth $u$, there holds
\cite[Lemma 2.3, pp. 73]{KilbasSrivastavaTrujillo:2006}
\begin{equation} \label{semig}
{_0\hspace{-0.3mm}I_x^{\gamma+\delta}}u={_0\hspace{-0.3mm}I_x^\gamma} {_0\hspace{-0.3mm}I_x^\delta} u.
\end{equation}

The fractional derivatives $\DDC0\beta$ and $\DDR0\beta$
are well defined for functions in $C^n[0,1]$ and are related to each other by the formula (cf.
\cite[pp. 91, eq. (2.4.6)]{KilbasSrivastavaTrujillo:2006})
\begin{equation} \label{eqn:rl=c}
 \DDC0\beta u =  {\DDR0\beta u}-\sum_{k=0}^{n-1} \frac
{u^{(k)}(0)} {\Gamma(k-\beta  +1)} x^{k-\beta}.
\end{equation}

The right-sided versions of fractional-order integrals and derivatives are defined analogously, i.e.,
\begin{equation*}
  ({_x\hspace{-0.3mm}I^\gamma_1} f) (x)= \frac 1{\Gamma(\gamma)}\int_x^1 (t-x)^{\gamma-1}f(t)\,dt,
\end{equation*}
and
\begin{equation*}
\begin{aligned}
  \DDC1\beta u = (-1)^n {_x\hspace{-0.3mm}I^{n-\beta}_1}\bigg(\frac {d^nu} {d x^n }\bigg),
\quad \quad  \DDR1\beta u =(-1)^n\frac {d^n} {d x^n} \bigg({_x\hspace{-0.3mm}I^{n-\beta}_1} u\bigg) .
\end{aligned}
\end{equation*}

We also recall the following useful change of integration order formula \cite[Lemma 2.7, part (a)]{KilbasSrivastavaTrujillo:2006}:
\begin{equation}\label{eqn:intpart}
({_0\hspace{-0.3mm}I_x^\beta} \phi,\varphi) = (\phi,{_x\hspace{-0.3mm}I_1^\beta} \varphi),
\qquad \hbox{ for all } \phi,\varphi\in L^2\II.
\end{equation}

The starting point of the variational formulations is the following theorem \cite{JinLazarovPasciak:2013a}.
\begin{theorem}\label{thm:difint}
The operators $\DDR0\beta $ and $_0I_x^\beta$ satisfy the following properties.
\begin{itemize}
  \item[(a)] The operator $\DDR0\beta$ defined on $\Cd 0^\beta\II$
extends continuously to an operator from $\Hdi 0 \beta$ to $L^2\II$.
  \item[(b)]For any $s,\beta\geq 0$, the operator $_0I_x^\beta $ is bounded from $\Hd s \II$ into $\Hdi 0
{s+\beta}$.
\end{itemize}
\end{theorem}

It follows directly from Theorem \ref{thm:difint}(b) that for $0<\beta<1$, the operator $\DDC0\beta$
extends continuously from $H^1\II$ into $H^{1-\beta}\II$. Further, by \eqref{eqn:rl=c} and
Theorem~\ref{thm:difint}(a), the following lemma holds \cite[Lemma 4.1]{JinLazarovPasciak:2013a}.

\begin{lemma} \label{lem:caprl}
For $u\in \Hdi 0 1$  and $\beta\in (0,1)$,
$\DDR 0\beta u = {_0I_x^{1-\beta}}(u^\prime):=\DDC 0\beta u$.
\end{lemma}

\subsection{Derivation of variational formulations}
Now we can derive the variational formulation constructively. We shall first construct the strong solutions
(in the case of $q=0$), and then verify that the strong solution satisfies a certain variational formulation.
The well-posedness of the variational formulations will be discussed in Section \ref{ssec:reg}.

We first consider the Riemann-Liouville case. For $f\in L^2\II$, we set $g={_0I_x^\alpha} f \in \Hdi 0\alpha$.
By Theorem~\ref{thm:difint}, the fractional derivative $\DDR0\alpha g$ is well defined. Now by the semigroup
property \eqref{semig}, we deduce
\begin{equation*}
{_0I_x^{2-\alpha}} g = {_0I_x^2}f\in \Hdi 0 2.
\end{equation*}
It is straightforward to check that $({_0I_x^2}f)^{\prime\prime}=f$ holds for smooth $f$ and hence also on
$L^2\II$ by a density argument. This implies that $\DDR0\alpha g=f$. We thus find that
\begin{equation} \label{strongrl}
u= -{_0I_x^\alpha} f +({_0I_x^\alpha} f)(1) x^{\alpha -1}
\end{equation}
is a solution of \eqref{strongp} in the Riemann-Louiville case (for $q=0$) since it satisfies the boundary
conditions and $\DDR0\alpha x^{\alpha-1}=(c_\alpha x)^{\prime\prime}=0$. Taking $v\in C_0^\infty\II$, Lemma~\ref{lem:caprl} implies
\begin{equation}\label{rlvar}
\begin{aligned}
  (\DDR0\alpha u,v) &= -\big(({_0I^{2-\alpha}_x} g)^{\prime\prime},v\big) =\big(({_0I^{2-\alpha}_x} g)^{\prime},v^\prime\big) \\
   &=\big({\DDR0{\alpha-1}} g,v^\prime\big) = ({_0I_x^{2-\alpha}}g',v'),
\end{aligned}
\end{equation}
where we have used the identity $\DDR0{\alpha} x^{\alpha-1}=0$ in the first step. Now the semigroup
property \eqref{semig} and the change of integration order formula \eqref{eqn:intpart} yield
\begin{equation}\label{rl-first}
   (\DDR0\alpha u,v)= \big({_0I^{1-\alpha/2}_x} g',{_xI_1^{1-{\alpha/2}}}v^\prime\big).
\end{equation}
Since $g\in \Hdi0{\alpha}$ and $v\in\Hd 1\II$, we can apply Lemma~\ref{lem:caprl} again to conclude
\begin{equation*}
  (\DDR0\alpha u,v) = -(\DDR0{\alpha/2} g,\ {\DDR1{\alpha/2}}v).
\end{equation*}
Further direct computation shows that $(\DDR0{\alpha/2} x^{\alpha-1},{\DDR1{\alpha/2}}v) = (\DDR0{\alpha-1}x^{\alpha-1},v')=0$.
Consequently,
\begin{equation}\label{adef}
   A(u,v):=-(\DDR0\alpha u,v)=-(\DDR0{{\alpha/2}}u,\,\DDR1{{\alpha/2}}v).
\end{equation}
Thus, $u$ is a solution of the variational problem: Find $u\in U:=\Hd {\alpha/2} \II$ such that
\begin{equation*}
A(u,v)=(f,v),\qquad \hbox{ for all } v\in V= U.
\end{equation*}
When $q\neq 0$, the variational problem becomes
\begin{equation}\label{var:rl}
a(u,v):= A(u,v)+(qu,v)=(f,v),\qquad \hbox{ for all } v\in V.
\end{equation}
We shall show in Theorem~\ref{thm:regrl} below that, under further assumptions,
there is a unique weak solution $u$ to \eqref{var:rl} when $q\in
L^\infty\II$.  In this case, we set $g={_0I_x^\alpha}(f-qu)$ and find that
\begin{equation}\label{strongrl2}
u(x)= -g(x) +g(1) x^{\alpha -1}
\end{equation}
is the unique solution of the variational equation.

The case of the Caputo derivative is similar and we only illustrate the derivation when $q=0$. Again we
first construct a ``strong'' solution. By the identity \eqref{eqn:rl=c}, both $\DDC0\alpha$ and
$\DDR0\alpha$ extend continuously to bounded operators on $\Hdi 0 {\alpha+\beta}$ for any
$\beta$ such that $\alpha+\beta>3/2$ and since they coincide on $\Cd0^{\alpha+\beta}\II$,
they coincide on $\Hdi 0 {\alpha+\beta}$. Thus if $f\in \Hd \beta\II$ then Theorem~
\ref{thm:difint}(b) implies that the function $g={_0I_x^\alpha f}$ is in  $\Hdi 0 {\alpha+\beta}$ and hence
\begin{equation*}
u(x) = -g(x) + g(1) x
\end{equation*}
represents a strong solution in the Caputo case. It is easy to see that $g(1)=u'(0)$.
Proceeding as in \eqref{rlvar}, for a smooth $v$ with $v(1)=0$, there holds
\begin{equation*}
\begin{aligned}
  \big(\DDC0\alpha u,v\big) &= -\big(\DDR0\alpha g,v\big)
  =\big(\DDR0{\alpha-1} g,v'\big)\\
&=\big(\DDR0{\alpha/2} u,\DDR1{\alpha/2} v\big) -u'(0)
 \big({\DDR0{\alpha-1}} x,v'\big).
\end{aligned}
\end{equation*}
The last term on the right hand side involves $u'(0)$, which is not allowed in a variational formulation in $U=\Hd{\alpha/2}\II$,
and can be removed by requiring the test function $v$ to satisfy $\big(\DDR0{\alpha-1} x,v^\prime\big)=c_\alpha
(x^{1-\alpha},v)=0.$ This leads to the same bilinear form as in the Riemann-Liouville case
except in the Caputo case, $V=\{v\in H^{\alpha/2}\II :\ v(1)=0, (x^{1-\alpha},v)=0\}.$

\subsection{Stability of the variational formulations}\label{ssec:reg}
Now we briefly discuss the stability of the variational formulations: find $u\in U$ such that
\begin{equation}\label{eqn:bilinform}
  a(u,v) = (f,v), \quad \forall v\in V,
\end{equation}
where $f$ belongs to either $L^2\II$ or suitable Sobolev space. Throughout we make the
following assumption on the bilinear form $a(u,v)$.
\begin{assumption}\label{ass:bilin}
The bilinear form $a(u,v)$  on $U\times V$ satisfies
\begin{itemize}
 \item[{$\mathrm{(a)}$}]The problem of finding $u \in U$ such that $a(u,v)=0$ for all $v \in V$
           has only the trivial solution $u\equiv 0$.
 \item[{$(\mathrm{a}^\ast)$}] The problem of finding $v \in V$ such that $a(u,v)=0$ for all $u \in U$
    has only the trivial solution $v\equiv 0$.
\end{itemize}
\end{assumption}

\begin{remark}
It can be verified \cite{ErvinRoop:2006,JinLazarovPasciak:2013a}
that in case of $q=0$, the bilinear form $a(u,v)$ is in fact coercive on $\Hd {\alpha/2} \II$ and hence
satisfies Assumption \ref{ass:bilin} in the Riemann-Liouville case. Hence it holds also for any bounded
nonnegative potential $q$. In the Caputo case, it can be verified directly that the
bilinear satisfies the assumption \ref{ass:bilin} for $q=0$, but for a general potential it is unclear.
\end{remark}

Then we have the following existence and stability result in the
case of the Riemann-Liouville derivative.

\begin{theorem}[Riemann-Liouville derivative] \label{thm:regrl}
Let Assumption \ref{ass:bilin} hold and $q\in L^\infty\II$. Then for any $f\in L^2\II$,
the variational problem \eqref{eqn:bilinform} has a unique weak solution $u\in U$. Further,
$u\in H^{\alpha-1+\beta}\II\cap \Hd{{\alpha/2}}\II$ for any $\beta\in[0,1/2)$ and satisfies
\begin{equation*}
  \|u \|_{H^{\al -1 +\beta}\II} \le c \|f\|_{L^2\II}.
\end{equation*}
\end{theorem}

\begin{proof}
We only sketch the idea and refer the full details to \cite{JinLazarovPasciak:2013a}.
Since the bilinear form $A(\cdot,\cdot)$ coerces the norm $\Hd  {{\alpha/2}} \II$,
Assumption \ref{ass:bilin}(a) implies that the bilinear form $a(\cdot,\cdot)$ satisfies
the inf-sup condition \cite[Theorem 4.3]{JinLazarovPasciak:2013a}. This and Assumption
\ref{ass:bilin}$(\mathrm{a}^\ast)$ imply that \eqref{eqn:bilinform} has a unique weak
solution $u\in\Hd{\alpha/2}\II$.  The unique solution $u$ satisfies \eqref{strongrl2} which
exhibits $u$ as a sum of functions, $-g\in \Hdi0 {\alpha} $ and $g(1) x^{\alpha-1} \in
\Hdi0{\alpha-1+\beta}$, for any $\beta \in [0,1/2)$.  This completes the sketch of the proof.
\end{proof}

\begin{remark}\label{rmk:regrl}
In general, the best possible regularity of the solution to \eqref{strongp} with a Riemann-Liouville fractional
derivative is $H^{\alpha-1+\beta}\II$ for any $\beta\in[0,1/2)$, due to the presence of the singular
term $x^{\alpha-1}$. The only possibility of an improved regularity is the case $({_0I_x^\alpha f})(1)=0$ (in case of $q=0$).
\end{remark}

We shall need also the adjoint problem in the Riemann-Liouville case:
given $f\in L^2\II$, find $w\in V$ such that
\begin{equation} \label{eqn:bilinadj}
   a(v,w)= ( v,f), \qquad \hbox{ for all } v\in U.
\end{equation}
Then there exists a unique solution $w\in U$ to the adjoint problem. Indeed, Assumption \ref{ass:bilin} implies
that the inf-sup condition for the adjoint problem holds. For $q=0$ and a right hand side $f\in L^2\II$, we have
\begin{equation*}
w= -{_xI_1^\alpha} f(x) +  ({_xI_1^\alpha} f)(0)(1-x)^{\alpha -1}.
\end{equation*}
This implies a similar regularity pickup, i.e., $w\in H^{\alpha-1+\beta}\II$.
Now we can repeat the arguments in the proof of Theorem \ref{thm:regrl} for a
general $q$ to deduce the regularity pick-up of the adjoint solution $w$.
Thus we have the following result.

\begin{theorem}\label{thm:regadjrl}
Let Assumption \ref{ass:bilin} hold and $q\in L^\infty\II$. Then for any $f\in L^2\II$,
there exists a unique weak solution $w\in \Hd{{\alpha/2}}\II$ to \eqref{eqn:bilinadj} such that for any
$\beta\in[0,1/2)$ there holds
\begin{equation*}
  \|w\|_{H^{\al-1+\beta}\II} \le c\|f\|_{L^2}.
\end{equation*}
\end{theorem}

Now we turn to the Caputo case, and have the following regularity result.
\begin{theorem}[Caputo derivative] \label{thm:regcap}
Let Assumption \ref{ass:bilin} hold and $q\in L^\infty\II$. Then for any  $f\in L^2\II$,
there exists a unique weak solution $u\in \Hd{\alpha/2}\II$ to \eqref{eqn:bilinform}.
Further, let $\beta\in[0,1]$ with $\alpha+\beta>3/2$ and $f\in \Hd\beta\II$, and
 $q\in L^\infty\II\cap H_0^\beta\II$. Then the weak solution $u$ solves \eqref{strongp} and
satisfies
\begin{equation*}
  \|u\|_{H^{\alpha+\beta}\II}\leq c\|f\|_{\Hd\beta\II}.
\end{equation*}
\end{theorem}
\begin{proof}
Again we only sketch the proof. First, by Assumption \ref{ass:bilin}, the bilinear form $a(u,v)$ satisfies
the inf-sup condition \cite[Theorem 4.5]{JinLazarovPasciak:2013a}, and thus problem \eqref{eqn:bilinform}
has a unique weak solution $u\in \Hd{\alpha/2}\II$. To derive
the regularity, we first consider the case $q=0$. Then direct computation shows that the function
\begin{equation}\label{eqn:caprep}
   u=({_0I_x^\alpha}f)(1)x-{_0I_x^\alpha f}(x)
\end{equation}
satisfies the differential equation under the condition $f\in \Hd{\beta}\II$ with $\alpha+\beta>3/2$, and the
boundary condition, and thus it is the solution to \eqref{strongp}. The first term in the representation is smooth,
and the second term belongs to $H^{\alpha+\beta}\II$, and thus $u\in H^{\alpha+\beta}\II$. In the case of a general $q$,
we rewrite the equation as $-\DDR0\alpha = \widetilde{f}$, with $\widetilde{f}=f-qu\in \Hd{\min(\beta,\alpha/2)}\II$
under the given condition. This together with a standard bootstrap argument concludes the proof of the theorem.
\end{proof}

\begin{remark}
The representation \eqref{eqn:caprep}  of the strong solution in the Caputo case is valid only under
the condition $f\in \Hd\beta\II$ with $\alpha+\beta>3/2$. If it is not satisfied, there is no known
solution representation. Note also the drastic difference in the solution regularity for the
Riemann-Liouville and Caputo cases: for the former, the solution is generally at best in $H^{\alpha-1+
\beta}\II$, for any $\beta\in[0,1/2)$, irrespective of the smoothness of the source $f$, whereas for
the latter, it can be made arbitrarily smooth if the potential $q$ and the source $f$ are sufficiently smooth.
The condition $q\in H_0^\beta\II \cap L^\infty\II$ in Theorem \ref{thm:regcap} can be relaxed to $q\in H^\beta
\II\cap L^\infty\II$ if $\beta\neq1/2$.
\end{remark}

Finally we turn to the adjoint problem in the Caputo case: find $w\in V$ such that
\begin{equation*}
  a(v,w) = (v, f)\quad \mbox{ for all } v \in U,
\end{equation*}
for some $f\in L^2\II$. The strong form reads $-\DDR1\alpha w + qw = f$, with $w(1)=0$ and $(x^{1-\alpha},w)=0$.
By repeating the preceding arguments, we deduce that for the case $q=0$, the solution $w$ can be written into
\begin{equation*}
  w =  - {_xI_1^\alpha} f + \frac{(x,f)}{\Gamma(\alpha)}(1-x)^{\alpha-1}.
\end{equation*}
Therefore, we have the following regularity estimate.
\begin{theorem}\label{thm:regcapadj}
Let Assumption \ref{ass:bilin} hold and $q\in L^\infty\II$. Then for
 $f\in L^2\II$ the solution $w$ to \eqref{eqn:bilinadj} is in $H^{\alpha-1+\beta}\II$
for any $\beta\in[0,1/2)$, and satisfies
\begin{equation*}
  \|w\|_{H^{\al -1 +\beta}\II} \le c \|f\|_{L^2\II}.
\end{equation*}
\end{theorem}

\section{Finite Element Method}\label{sec:fem}

Now we turn to the finite element formulation of the eigenvalue problem \eqref{eigen-strong}, based
on variational formulations described in Section \ref{sec:prelim}, and finite-dimensional subspaces
$U_h\subset U$ and $V_h\subset V$. Then the details of its efficient implementation will be presented.
Finally, we shall apply the abstract convergence theory due to Buba\v{s}ka-Osborn
\cite{Babuska-Osborn} to derive preliminary error estimates for the approximate eigenvalues.
\subsection{Finite element spaces}
In the finite element method, we first divide the unit interval $D$ into  a (not
necessarily uniform) mesh, with the grid points $0=x_0<\ldots<x_{m+1}=1$. We
denote by $h_i=x_i-x_{i-1}$, $i=1,\ldots,m+1$, the local mesh sizes, and $h=
\max_i h_i$ the (global) mesh size. Then we define the finite dimensional space
$$
X_h=\{v: ~~v \in C(\overline{D}),~~\mbox{linear on} ~~[x_{i-1},x_i],~~ i=1, \dots, m+1\}.
$$
The nodal basis for $X_h$ will be the standard ``hat functions'', denoted by $\phi_j(x)$, $j=0, \dots,m+1$.
Then the solution finite element space $U_h$ is defined as
\begin{equation}\label{eqn:Uh}
U_h =\{ v \in X_h: ~~ v(0)=v(1)=0\}.
\end{equation}

The test space $V_h$ depends on the type of the fractional derivative. It can be taken to be
$V_h=U_h$ for the Riemann-Liouville derivative. For the Caputo derivative, we define
a finite element subspace $V_h\subset V$ as
\begin{equation}\label{eqn:Vh}
   V_h=    \{v\in X_h: v(1)=0, ~~~(x^{1-\alpha}, v)=0\}.
\end{equation}
To construct a basis for $V$, we consider the candidate set $\{\widetilde{\phi}_j\}_{j=0}^m$,
with $\widetilde{\phi}_j=\phi_j-\gamma_j(1-x)$, where the constants $\{\gamma_j\}$ are determined
to satisfy the integral constraint $(x^{1-\alpha}, \widetilde{\phi}_j)=0$, i.e.,
$\gamma_i= \int_0^1x^{1-\alpha}\phi_i(x)dx/ \int_0^1(1-x)x^{1-\alpha}dx.$
Since $\int_0^1(1-x)x^{1-\alpha}dx={\frac{1}{2-\alpha}-\frac{1}{3-\alpha}}$ and for $i=1,\ldots,m$,
\begin{equation*}
  \begin{aligned}
    \int_0^1x^{1-\alpha}\phi_i(x)dx
        & = \int_{x_{i-1}}^{x_i}x^{1-\alpha}\frac{x-x_{i-1}}{h_i}dx + \int_{x_i}^{x_{i+1}}x^{1-\alpha}\frac{x_{i+1}-x}{h_{i+1}}dx\\
        &=\left(\frac{1}{2-\alpha}-\frac{1}{3-\alpha}\right)\left(\frac{x_{i-1}^{3-\alpha}}{h_{i}}+\frac{x_{i+1}^{3-\alpha}}{h_{i+1}}
          -\frac{x_i^{3-\alpha}}{h_i}-\frac{x_i^{3-\alpha}}{h_{i+1}}\right),
  \end{aligned}
\end{equation*}
the coefficient $\gamma_i$ is given by
\begin{equation*}
\gamma_i = \frac{x_{i-1}^{3-\alpha}}{h_i}+\frac{x_{i+1}^{3-\alpha}}{h_{i+1}}
-\frac{x_i^{3-\alpha}}{h_i}-\frac{x_i^{3-\alpha}}{h_{i+1}}, \quad i=1,\ldots,m.
\end{equation*}
Similarly, the coefficient $\gamma_0$ is given by $\gamma_0=h_1^{2-\alpha}$.
In particular, on a uniform mesh, i.e., $h_i=h$ and $x_i=ih$, the expression for $\gamma_i$
simplifies to $\gamma_i = h^{2-\alpha}((i-1)^{3-\alpha}+(i+1)^{3-\alpha}-2i^{3-\alpha})$.
By definition, the sequence $\{\gamma_i\}$ is strictly positive.
Clearly the set $\{\widetilde{\phi}_j\}_{j=0}^m$ spans the subspace $V_h$, however, the functions
$\widetilde{\phi}_j, j=0, \dots, m$ are linearly dependent. To see this, we observe that by the identity
$\sum_{j=0}^m(1-x_j)\phi_j(x) = 1-x$, there holds $\sum_{j=0}^m(1-x_j)\gamma_j = 1$. Thus,
\begin{equation*}
   \sum_{j=0}^m (1-x_j)\widetilde{\phi}_j = \sum_{j=0}^m (1-x_j)\phi_j(x) - (1-x)\sum_{j=0}^m\gamma_j(1-x_j)=0,
\end{equation*}
i.e., $\widetilde{\phi}_0 = -\sum_{j=1}^m(1-x_j)\widetilde{\phi}_j(x)$. In our computation,
we use the basis set $\{\widetilde{\phi}_j\}_{j=1}^m$.

These finite element spaces $U_h$ and $V_h$ satisfy the following approximation
properties \cite[Lemma 5.1]{JinLazarovPasciak:2013a}.
\begin{lemma}\label{lem:fem}
Let the mesh ${\mathcal T}_h$ be quasi-uniform. If $u \in H^\gamma\II \cap \Hd {\al/2}\II$ with
$ \alpha/2 \le \gamma \le 2$, then
\begin{equation*}
\inf_{v \in U_h} \| u -v \|_{H^{\al/2}\II} \le c h^{\gamma -\alpha/2} \|u\|_{H^\gamma\II}.
\end{equation*}
Further, if $u \in H^\gamma\II\cap V$, then
\begin{equation*}
\inf_{v \in V_h} \|u -v\|_{H^{\alpha/2}\II} \le c h^{\gamma -{\al/2}} \|u\|_{H^\gamma\II}.
\end{equation*}
\end{lemma}

\subsection{Implementation details}
With the finite dimensional subspaces $U_h \subset U$ and $V_h \subset V$ at hand, we
can now develop the finite element formulation of the eigenvalue problem \eqref{eigen-strong}:
find $u_h \in U_h$ and $\lambda_h \in \CC$ such that
\begin{equation*}
   a(u_h,v_h)= \lambda_h (u_h,v_h) \quad \forall v_h \in V_h.
\end{equation*}
Upon expanding the approximate eigenfunction $u_h$ into the canonical basis $\phi_j$ (with
$\mathbf{u}$ being the expansion coefficient vector) and choosing $v_h=\phi_i$ for the
Riemann-Liouville derivative (respectively, $v_h=\widetilde{\phi}_i$ for the Caputo derivative),
we arrive at the the following finite dimensional generalized eigenvalue problem
\begin{equation}\label{eqn:eigfem}
   \mathbf{Au} = \lambda_h\mathbf{Mu},
\end{equation}
where the stiffness matrix $\mathbf{A}=[a_{ij}]\in\mathbb{R}^{m\times m}$ and the mass matrix
$\mathbf{M}=[b_{ij}]\in\mathbb{R}^{m\times m}$ are respectively defined by
\begin{equation*}
  a_{ij}= \left\{\begin{aligned}
      (\DDR0{\alpha/2} \phi_j,\,\DDR1{\alpha/2}\phi_i)+(q\phi_j,\phi_i), &\quad \text{Riemann-Liouville case},\\
      (\DDR0{\alpha/2} \phi_j,\,\DDR1{\alpha/2}\widetilde{\phi}_i)+(q\phi_j,\widetilde{\phi}_i), & \quad \text{Caputo case},
    \end{aligned}\right.
\end{equation*}
and
\begin{equation*}
b_{ij}= \left\{\begin{aligned}
      (\phi_j,\,\phi_i),&\quad \text{Riemann-Liouville case},\\
      (\phi_j,\,\widetilde{\phi}_i), & \quad \text{Caputo case}.
    \end{aligned}\right.
\end{equation*}

Next we give explicit formulas for computing the stiffness matrix $\mathbf{A}$ and mass matrix $\mathbf{M}$.
We first consider the Riemann-Liouville case. We note that here the computation of the integrals involving
the potential term and the mass matrix is rather straightforward. Hence we shall focus our derivation on the
leading term, and for any order $\gamma=\alpha/2\in(1/2,1)$, we evaluate
\begin{equation}\label{eqn:intriem}
  -(\DDR0 \gamma \phi_j,\, \DDR 1 \gamma\phi_i),
\end{equation}
where $\phi_i$ and $\phi_j$ are the hat basis functions associated to the $i$th and
$j$th interior grid points, respectively. We note that $\phi_i'$ and $\phi_j'$ are
both piecewise constant, each with a support on two neighboring elements. Without loss of generality,
we can consider an interval $[x_{k-1}, x_k]$. We introduce the following two functions:
\begin{equation*}
  I_k(x)= 
    {_0I_x^{1-\gamma}}\chi_{[x_{k-1},x_k]}
\quad\mbox{  and  }\quad
\widetilde{I}_k(x)={_xI_1^{1-\gamma}}\chi_{[x_{k-1},x_k]}, 
\end{equation*}
where $\chi_S(x)$ refers to the characteristic function of the set $S$.
 Simple calculations yield (with $c_\gamma=1/\Gamma(2-\gamma)$)
\begin{equation*}
I_k(x)=\left\{\begin{array}{ll}
    0, &\quad x\leq x_{x_{k-1}},\\
    c_\gamma(x-x_{k-1})^{1-\gamma}, &\quad x_{k-1} \leq x\leq x_k,\\
    c_\gamma\left((x-x_{k-1})^{1-\gamma}-(x-x_k)^{1-\gamma}\right), & \quad x>x_k,
  \end{array}\right.
\end{equation*}
and
\begin{equation*}
  \widetilde{I}_k(x)= \left\{\begin{array}{ll}
    0, &\quad x\geq x_{x_k},\\
    c_\gamma(x_k-x)^{1-\gamma}, &\quad x_{k-1} \leq x\leq x_k,\\
    c_\gamma\left((x_k-x)^{1-\gamma}-(x_{k-1}-x)^{1-\gamma}\right), & \quad x<x_{k-1}.
  \end{array}\right.
\end{equation*}
To further simplify the notation, we introduce the functions $f_k(x)$ and $g_k(x)$ by
\begin{equation*}
   f_k(x) = c_\gamma(x-x_k)^{1-\gamma}\chi_{[x_k,1]} \quad \text{and}\quad
   g_k(x)  = c_\gamma(x_k-x)^{1-\gamma}\chi_{[0,x_k]}.
\end{equation*}
Consequently, we can succinctly express $I_k$ and $\tilde{I}_k$ as
\begin{equation*}
      I_k(x)=f_{k-1}(x)-f_k(x)\quad\mbox{and} \quad
      \widetilde{I}_k(x)=g_{k}(x)-g_{k-1}(x).
\end{equation*}

Now we can evaluate the integral \eqref{eqn:intriem}. For each interior node $x_i$,
$\phi_i'=\frac{1}{h_i}\chi_{[x_{i-1},x_i]}-\frac{1}{h_{i+1}}\chi_{[x_i,x_{i+1}]}$,
($h_i=x_i-x_{i-1}$ is the local mesh size). Hence, for any two interior nodes
$x_i$ and $x_j$, there holds
\begin{equation}\label{eqn:riementry}
   -(\DDR0\gamma \phi_j,\,\DDR1\gamma\phi_i)   = \int_0^1\left(\tfrac{1}{h_j}I_j-\tfrac{1}{h_{j+1}}I_{j+1}\right) \left(\tfrac{1}{h_i}\widetilde{I}_i-\tfrac{1}{h_{i+1}}\widetilde{I}_{i+1}\right)dx.
\end{equation}
Therefore, it suffices to evaluate
\begin{equation*}
\begin{split}
    \int_0^1I_j\widetilde{I}_idx & =\int_0^1 (f_{j-1}-f_j)(g_i-g_{i-1})dx \\
        & =\int_0^1 (f_{j-1}g_i -f_{j-1}g_{i-1} -f_jg_i+f_jg_{i-1}) dx.
\end{split}
\end{equation*}
Each of these four terms can be expressed in closed form using the Gamma function $\Gamma(\cdot)$
as follows: if $x_i\leq x_j$, $\int_0^1f_j(x)g_i(x) dx=0$, and if $x_i> x_j$
\begin{equation}\label{eqn:fg}
  \begin{aligned}
    \int_0^1f_j(x)g_i(x) dx&= c_\gamma^2\int_{x_j}^{x_i}(x-x_i)^{1-\gamma}(x_i-x)^{1-\gamma}dx\\
      &=c_\gamma^2(x_i-x_j)^{3-2\gamma}\int_0^1s^{1-\gamma}(1-s)^{1-\gamma}ds
      =\frac{(x_i-x_j)^{3-\alpha}}{\Gamma(4-\alpha)}.
  \end{aligned}
\end{equation}
Here the last line follows from the definition of Beta function and its relation to Gamma function, i.e.,
$$\int_0^1s^{1-\gamma}(1-s)^{1-\gamma}ds=B(2-\gamma,2-\gamma)=\frac{\Gamma(2-\gamma)\Gamma(2-\gamma)}{\Gamma(4-2\gamma)}.$$

The matrix $\mathbf{A}_0=\left[-(\DDR0{\alpha/2}\phi_j,\ \DDR1{\alpha/2}\phi_i)\right]\in\mathbb{R}^{m\times m}$ has the
following structure.
\begin{proposition}
The matrix $\mathbf{A}_0$ is of lower Hessenberg form, and on a uniform mesh, it is Toeplitz.
\end{proposition}
\begin{proof}
First we show that $\tilde{a}_{i,j}=-(\DDR0{\alpha/2}\phi_j,\
\DDR1{\alpha/2}\phi_i)$ vanishes if $j>i+1$. To this end, we note that the functions $I_k(x)$
and $\widetilde{I}_k(x)$ vanish on the interval $[0, x_{k-1}]$ and $[x_k,1]$, respectively.
Hence, for index $j>i+1$, in the integral \eqref{eqn:riementry}, the first term of the
integrand has a support $\text{supp}(\frac{1}{h_j}I_j-\frac{1}{h_{j+1}}I_{j+1})\subset [x_j,1]$,
and the second term has a support $\text{supp}(\frac{1}{h_i}\widetilde{I}_i-\frac{1}{h_{i+1}}
\widetilde{I}_{i+1})\subset [0,x_{i+1}]$, i.e., the integrand in \eqref{eqn:riementry} vanishes
identically for $j>i+1$, which shows the first assertion. According to \eqref{eqn:fg},
on a uniform mesh, the integral $\int_0^1f_jg_idx$ depends only on the difference of the indices,
i.e., $x_i-x_j=(i-j)h$, and so is the entry $\tilde{a}_{i,j}$. Hence the matrix $\mathbf{A}_0$ is Toeplitz.
\end{proof}

\begin{remark}
By the Toeplitz structure, on a uniform mesh, the matrix $\mathbf{A}_0$ can be constructed by computing
only the first row and first column and requires only two vectors (actually only $m+1$ nonzeros). Further,
solving linear systems with Toeplitz matrices can be carried out efficiently via fast Fourier transform.
\end{remark}

Now we turn to the Caputo derivative $\DDC 0 \alpha$. By the construction of the basis functions
$\{\widetilde{\phi}_j\}$, the leading term
\begin{equation*}
     - (\DDR0\gamma \phi_j,\, \DDR1 \gamma \widetilde{\phi}_i) = -(\DDR0\gamma \phi_j,\, \DDR1
     \gamma \phi_i) + \gamma_i(\DDR0\gamma\phi_j,\,\DDR1\gamma(1-x))
\end{equation*}
It follows directly from this identity that the stiffness matrix $\mathbf{A}_0^c\in\mathbb{R}^{m\times m}$ for
the Caputo case (with $q=0$) is a rank-one perturbation of the Riemann-Liouville stiffness matrix $\mathbf{A}_0$, i.e.,
\begin{equation*}
  \mathbf{A}_0^{c}=\mathbf{A}_0+\boldsymbol{\gamma}\mathbf{b}^\mathrm{t},
\end{equation*}
where the column vectors $\boldsymbol{\gamma}\in\mathbb{R}^m$ and $\mathbf{b}\in\mathbb{R}^m$ are given by
\begin{equation*}
    \boldsymbol{\gamma}  = (\gamma_1,\gamma_2,\ldots,\gamma_m)^\mathrm{t}\quad \mbox{and}\quad
    \mathbf{b} = (b_1,b_2,\ldots,b_m)^\mathrm{t}.
\end{equation*}
Further, the $j$th entry $b_j:=(\DDR0\gamma\phi_j,\,\DDR1\gamma(1-x))$ of the vector $\mathbf{b}$ can be explicitly evaluated as
\begin{equation*}
  \begin{aligned}
     b_j&=-({_0I_x^{1-\gamma}}(\tfrac{1}{h_j}\chi_{[x_{j-1},x_j]}-\tfrac{1}{h_{j+1}}\chi_{[x_j,x_{j+1}]}),\, {_xI_1^{1-\gamma}} 1)\\
        &= -(\tfrac{1}{h_j}\chi_{[x_{j-1},x_j]}-\tfrac{1}{h_{j+1}}\chi_{[x_j,x_{j+1}]},\, {_xI_1^{2-\alpha}}1)\\
        &= -\tfrac{1}{\Gamma(4-\alpha)}\left[\tfrac{(1-x_j)^{3-\alpha}}{h_j}+\tfrac{(1-x_j)^{3-\alpha}}{h_{j+1}}-
         \tfrac{(1-x_{j-1})^{3-\alpha}}{h_j}-\tfrac{(1-x_{j+1})^{3-\alpha}}{h_{j+1}}\right],
  \end{aligned}
\end{equation*}
where the second line follows from \eqref{eqn:intpart} and the semigroup property  \eqref{semig}.
Similarly, the matrix involving the potential term and the mass matrix in the Caputo case can be
constructed via a simple rank-one perturbation from the Riemann-Liouville counterparts.
\subsection{Error analysis for the eigenvalue problem}\label{ssec:FEM-eigen}
In this part, we provide preliminary error analysis of the finite element approximations.
We shall follow the notation and use some fundamental results from
\cite{Babuska-Osborn}. To this end, we introduce the operator $T: L^2\II\to H_0^{\alpha/2}\II$ defined by
\begin{equation}\label{T-operator}
T f \in H_0^{\alpha/2}\II, \quad a(Tf,v) = (f,v) \quad \forall v \in V.
\end{equation}
Obviously, $T$ is the solution operator of the source problem \eqref{strongp}. According to Theorems
\ref{thm:regrl} and \ref{thm:regcap}, the solution operator $T$ satisfies the following smoothing property:
$$
\| Tf \|_{H^{\alpha/2}\II} \le c\|f\|_{L^2\II}.
$$
Since the space $ H_0^{\alpha/2}\II$ is compactly embedded into $ L^2\II$ \cite{AdamsFournier:2003}, we deduce that
$T:~ L^2\II\to L^2\II$ is a compact operator. Meanwhile, by viewing $T$ as an operator on the space
$ H^{\alpha/2}\II$ and using the regularity pickup established in Theorems \ref{thm:regrl} and \ref{thm:regcap}
we can show that $T: ~~H^{\alpha/2}\II \to H^{\alpha/2}\II$ is compact. Then it follows immediately
from \eqref{T-operator} that $(\lambda, u)$ is an eigenpair
of \eqref{eigen-weak} if and only if
\begin{equation*}
 Tu= \lambda^{-1} u, \quad u \not = 0,
\end{equation*}
i.e., if and only if $ (\mu=1/\lambda, u)$ is an eigenpair of $T$. With the help of this correspondence,
the properties of the eigenvalue problem \eqref{eigen-weak} can be derived from the spectral theory for
compact operators \cite{Dunford-Schwartz}. Let $\sigma(T) \subset \CC$ be the set of all eigenvalues of
$T$ (or its spectrum), which is known to be a countable set with no nonzero limit points. Due to
Assumption \ref{ass:bilin} on the bilinear form $a(u,v)$, zero is not an eigenvalue of $T$. Furthermore,
for any $\mu \in \sigma(T)$, the space $N(\mu I -T)$, where $N$ denotes the null space, of eigenvectors
corresponding to $\mu$ is finite dimensional.

Now let $T_h: ~U_h \to U_h$ be a family of operators for $0<h<1$ defined by
\begin{equation}\label{Th-operator}
T_h f \in U_h, \quad a(T_hf,v) = (f,v)\quad \forall v \in V_h. 
\end{equation}

Next we apply the abstract approximation theory for the spectrum of variationally formulated eigenvalue problems
\cite[Section 8]{Babuska-Osborn}. To this end, we need to establish the approximation properties of the finite element
method \eqref{eigen-weak-h}. Let $\lambda^{-1} \in \sigma(T) $ be an eigenvalue of $T$ with algebraic
multiplicity $m$. Since $T_h \to T$ in norm, $m$ eigenvalues $\lambda^1_h, \dots, \lambda^m_h$ of $T_h$ will
converge to $\lambda$. The eigenvalues $ \lambda^j_h$ are counted according to the algebraic multiplicity of
$\mu_h^j=1/\lambda^j_h$ as eigenvalues of $T_h$. Further, we define the following finite dimensional spaces
\begin{equation*}
  \begin{aligned}
     M(\lambda)&=\left\{u:\ \|u\|_{H^{\alpha/2}\II}=1
       ~~\mbox{a generalized eigenvector of \eqref{eigen-weak} corresponding to} ~~\lambda\right\},\\
     M^*(\lambda)& =\left \{u:\ \|u\|_{H^{\alpha/2}\II}=1
       ~~\mbox{a generalized adjoint eigenvector of \eqref{eigen-weak} corresponding to} ~~\lambda\right\},
  \end{aligned}
\end{equation*}
and the following quantities:
\begin{equation*}
  \begin{aligned}
    \epsilon_h&=\epsilon_h(\lambda) = \sup_{u \in M(\lambda)} \inf_{v \in U_h} \|u-v\|_{H^{\alpha/2}\II},\\
    \epsilon^*_h&=\epsilon^*_h(\lambda) = \sup_{u \in M^*(\lambda)} \inf_{v \in V_h} \|u-v\|_{H^{\alpha/2}\II}.
  \end{aligned}
\end{equation*}

Then we have the following estimates on $\epsilon_h$ and $\epsilon_h^\ast$.
\begin{lemma}\label{lem:femeps}
Let Assumption \ref{ass:bilin} hold, $r\in (0,{\alpha/2}-1/2)$, and the finite element spaces $U_h$
and $V_h$ be defined by \eqref{eqn:Uh} and \eqref{eqn:Vh}, respectively. Then the following error
bounds for $\epsilon_h$ and $\epsilon_h^\ast$ are valid.
\begin{itemize}
  \item[(a)] For Caputo derivative: if  $q\in H_0^s\II\cap L^\infty\II$, $0\le s\le1$, such that
   $\alpha+s>3/2$, then $\epsilon_h \le c h^{\min(\alpha+s,2)-\alpha/2}$ and $\epsilon^*_h \le C h^r$.
  \item[(b)] For Riemann-Liouville derivative: $\epsilon_h \le c h^r$ and $\epsilon^*_h \le C h^r$.
\end{itemize}
\end{lemma}
\begin{proof}
The needed regularity for the Caputo case is a simple consequence of
Theorems \ref{thm:regcap}. We consider the only the case $q\in H_0^s\II\cap L^\infty\II$.
Since the right hand side of the problem \eqref{eigen-strong} is $\lambda u \in\Hd{\alpha/2}\II$, we have
$ u \in \Hd {\alpha/2}\II\cap H^{\alpha+s}\II$ in view of Theorem \ref{thm:regcap}. Next
let $\Pi_h u \in U_h$ be the finite element interpolant of $u(x)$. Then by Lemma \ref{lem:fem}
we deduce that
\begin{equation*}
\| u - \Pi_h u \|_{\Hd \gamma\II} \le h^{\min(\alpha+s,2)-\gamma} \|u\|_{H^{\min(\alpha+s,2)}\II}
\end{equation*}
for any $0 \le \gamma \le 1$. Therefore, by $\|u\|_{H^{\alpha+s}\II} \le c$, we get
\begin{equation*}
\epsilon_h = \sup_{u \in M(\lambda)} \inf_{v \in U_h} \|u-v\|_{H^{\alpha/2}\II} \le
\sup_{u \in M(\lambda)}\|u- \Pi_h u\|_{H^{\alpha/2}\II} \le c h^{\min(\alpha+s,2)-{\alpha/2}}.
\end{equation*}
The estimate $\epsilon_h^\ast$ follows from Theorem \ref{thm:regcapadj} and Lemma \ref{lem:fem}.
The Riemann-Liouville case follows analogously from regularity estimates for the source
problem and adjoint source problem in Theorems \ref{thm:regrl} and \ref{thm:regadjrl}, respectively.
\end{proof}

Now we state the main result for the approximation error of the eigenvalues of problem
\eqref{eigen-strong}. It follows immediately from \cite[Theorem 8.3 and pp. 683--714]{Babuska-Osborn}
and Lemma \ref{lem:femeps}.
\begin{theorem}\label{eigencval-error}
Let Assumption \ref{ass:bilin} hold. For $\mu \in \sigma(T)$, let $\delta$ be the smallest integer $k$ such that $N((\mu-T)^k) =
N((\mu -T)^{k+1})$. Suppose that for each $h$ there is unit vector $w_h^j$ satisfying
$((\lambda_h^j)^{-1} - T_h)^k w_h^j=0$, $j=1,\dots, m$ for some integer $k \le \delta$.
\begin{itemize}
  \item[(a)] For Caputo derivative, if $q\in H_0^s\II\cap
   L^\infty\II$ for some $0\le s\le1$ and $\alpha+s>3/2$, then for any
   $\gamma<\min(\alpha+s,2)-1/2$, there holds $|\lambda - \lambda_h^j| \le c h^{\gamma/\delta}$.
  \item[(b)] For Riemann-Liouville derivative, if $q\in L^\infty\II$, then for any $\gamma<\alpha-1$, there holds $|\lambda - \lambda_h^j| \le ch^{\gamma/\delta}$.
\end{itemize}
\end{theorem}
\begin{remark}
In the case of a zero potential $q \equiv0$, it is known that the eigenvalues $\lambda$ are
zeros of the Mittag-Leffler functions $E_{\alpha,\alpha}(-\lambda)$ and $E_{\alpha,2}(-\lambda)$
for the Riemann-Liouville and Caputo case, respectively. This connection together with the
exponential asymptotics of Mittag-Leffler functions allows one to show that eigenvalues with sufficiently
large magnitudes are simple \cite{Djrbashian:1993,Sedletskii:2000,JinRundell:2012}, and hence
the multiplicity $\delta=1$. In our computations we observed that for all potential $q$
the eigenvalues are simple, i.e., $\delta=1$.
\end{remark}

\begin{remark}
The (theoretical) rate of convergence for the case of Riemann-Liouville fractional derivative is lower than
that for the Caputo case. This is due to limited smoothing property for the Riemann-Liouville fractional
derivative operator. It naturally suggests that an adaptive procedure involving a proper grid refinement
or an augmentation of the solution and test spaces should be used. Nonetheless, a second-order
convergence is observed for the eigenvalue approximations, even though the eigenfunction approximations
in $H^{\alpha/2}\II$-norm are less accurate due to the conceived singularity, especially for $\alpha$ close
to unity. In other words, the abstract theory gives only suboptimal convergence rates.
\end{remark}

\section{Numerical experiments}\label{sec:numerics}

In this part we present numerical experiments for the FSLP to illustrate the performance  of the
finite element method. We consider the following three different potentials:
\begin{itemize}
  \item[(a)] a zero potential $q_1(x)=0$;
  \item[(b)] a smooth potential $q_2(x)=20x^3(1-x)e^{-x}$;
  \item[(c)] a discontinuous potential $q_3(x)=-2x\chi_{[0,1/5]}+(-4/5+2x)\chi_{[1/5,2/5]}+\chi_{[3/5,4/5]}$.
\end{itemize}
The potentials $q_1$ and $q_2$ are smooth and belong to the space $H^1_0\II$, while the potential $q_3$
is piecewise smooth and it belongs to the space $H_0^s\II$ for any $s<1/2$. The profiles of
the potentials $q_2$ and $q_3$ are shown in Fig. \ref{fig:pot}.  These examples are used to
illustrate the influence of the potential on the convergence behavior of the finite
element method.

As was mentioned in Section \ref{sec:introduction}, in the case of a zero potential $q(x)\equiv 0$,
the eigenvalues are known to be the zeros of the Mittag-Leffler function $E_{\alpha,2}(-\lambda)$
for the Caputo derivative (respectively $E_{\alpha,\alpha}(-\lambda)$ for the Riemann-Liouville
derivative). This can be numerically verified directly, cf. Fig. \ref{fig:mitlef}. Nonetheless,
there is no known accurate solver for locating these zeros since accurately evaluating Mittag-Leffler
functions is highly nontrivial, and further, it does not cover the interesting case of a general $q$.

\begin{figure}[h!]
  \centering
  \includegraphics[trim = .5cm 0cm 1cm 0cm, clip=true,width=7cm]{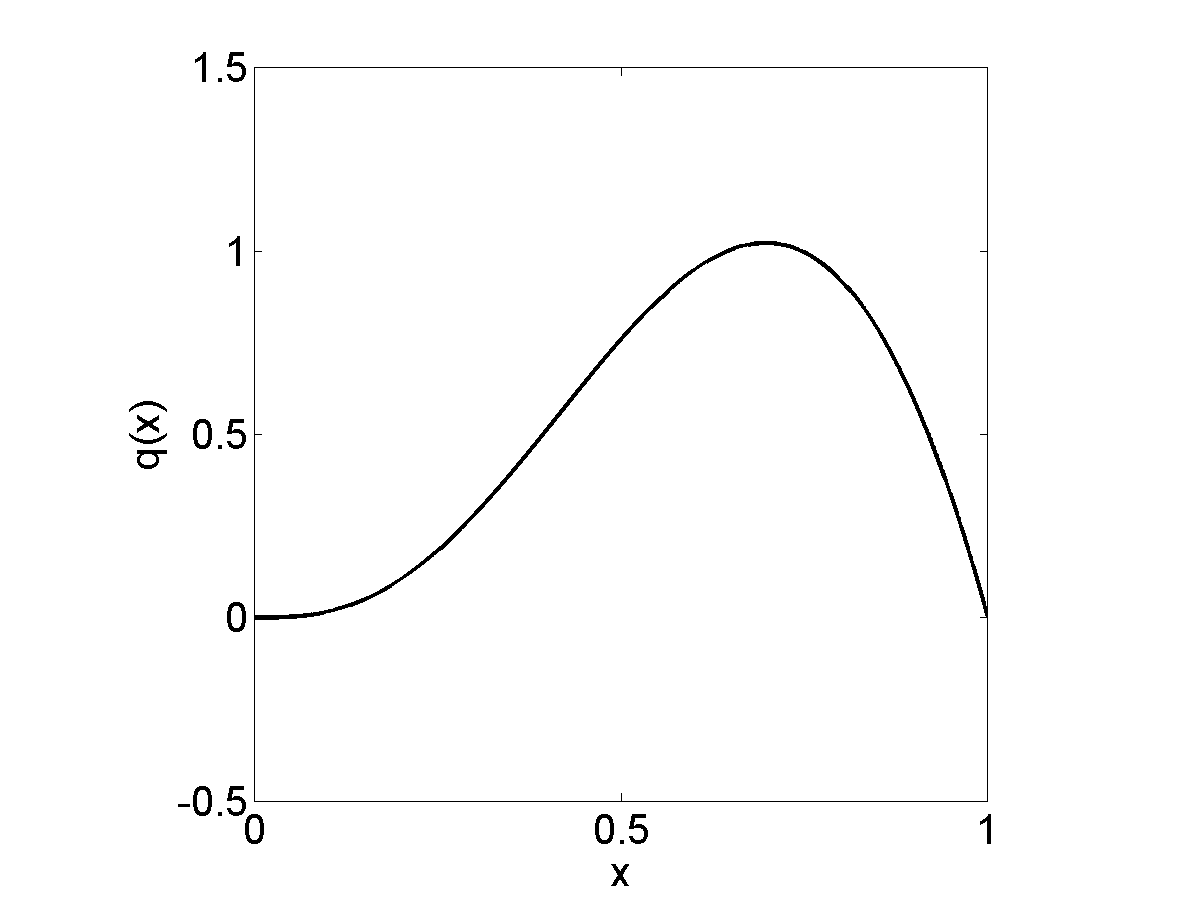}
  \includegraphics[trim = .5cm 0cm 1cm 0cm, clip=true,width=7cm]{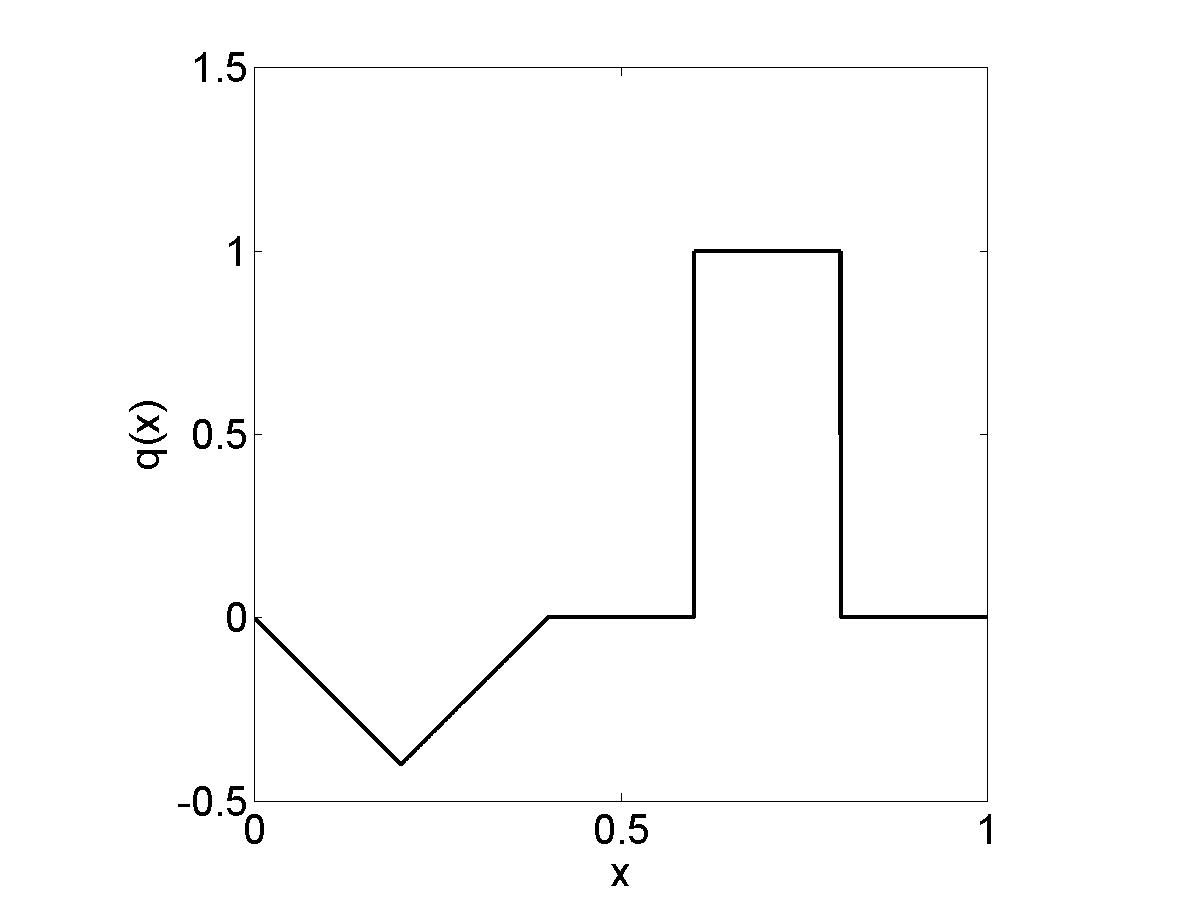}
  \caption{The profile of the potentials $q_2$ (left) and $q_3$ (right).}\label{fig:pot}
\end{figure}

\begin{figure}[h!]
  \centering
  \begin{tabular}{cc}
  \includegraphics[trim = 1cm 0cm 1cm 0cm, clip=true,width=7cm]{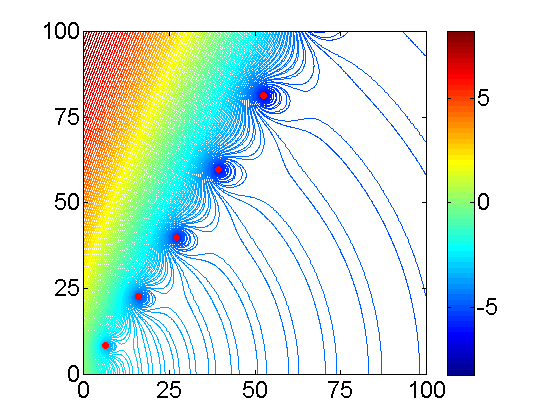}&
  \includegraphics[trim = 1cm 0cm 1cm 0cm, clip=true,width=7cm]{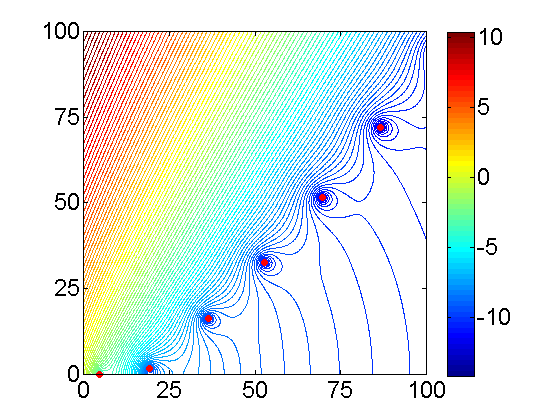}\\
     Caputo case & Riemann-Liouville case
  \end{tabular}
  \caption{Zeros of Mittag-Leffler function $E_{\alpha,2}(-\lambda)$ and the variational eigenvalues
  for $q=0$ for $\alpha=4/3$. The red dots are the eigenvalues computed by the finite element method, and
  the contour lines plot the value of the functions $\log|E_{\alpha,2}(-\lambda)|$ (left) and
  $\log|E_{\alpha,\alpha}(-\lambda)|$ (right) over the complex domain. The red dots lie in the
  deep blue wells, which correspond to zeros of the Mittag-Leffler functions.}\label{fig:mitlef}
\end{figure}

Throughout we partition the domain into a uniform mesh with $m$ equal subintervals,
i.e., a mesh size $h=1/m$. We measure the accuracy of an FEM approximation $\lambda_h$,
by the absolute error, i.e., $e(\lambda_h) = |\lambda-\lambda_h|$ and the reference value
$\lambda$ is computed on a very refined mesh with $m=10240$ and checked against that computed by
the quasi-Newton method developed in \cite{JinRundell:2012}. All the computations were performed using
\texttt{MATLAB R2009c} on a laptop with a dual-core $6.00$G RAM memory. The generalized eigenvalue
problems \eqref{eqn:eigfem} were solved by built-in \texttt{MATLAB} function \texttt{eigs} with a
default tolerance. Below we shall discuss the cases of Caputo and Riemann-Liouville derivative separately,
since their eigenfunctions have very different regularity and hence, one naturally expects that the
approximations exhibit different convergence behavior. Finally, we discuss possible extensions
and as an application of the finite element method, we study also the behavior of the fractional SLP with a
Riemann-Liouville derivative.

\subsection{Caputo derivative case}

By the solution theory in Section \ref{sec:prelim}, for the smooth potentials $q_1$ and $q_2$, the
eigenfunctions are in the space $H^{\al+1}\II\cap \Hd{\alpha/2}\II$, whereas for the discontinuous
potential $q_3$, the eigenfunctions lie in the space $H^{\alpha+s}\II\cap \Hd{\alpha/2}\II$ for
any $s\in[0,1/2)$. Hence they can be well approximated by uniform meshes. Further, these enhanced
regularity estimates predict a convergence rate of order $\min(\alpha+s,2)-1/2$ (with $s\in [0,1]$
being an exponent such that $q\in H_0^\beta\II$) for the approximate eigenvalues. In
particular, theoretically, the best possible convergence rate is $O(h^{3/2})$.

In Tables \ref{tab:al66capq1}-\ref{tab:al66capq3} we present the errors of the first ten eigenvalues for
$\alpha=5/3$ and different mesh sizes, where the empirical convergence rates are also shown. For all three
potentials, there are only two real eigenvalues, and the rest appears as complex conjugate pairs. All the computed
eigenvalues are simple.
Numerically, a second-order convergence is observed for all the eigenvalues. Further, the presence of a
potential term influences the errors very little: they are almost identical for all three potentials, as are the convergence rates.
The empirical rate is at least one half order higher than the theoretical one. The mechanism of
the ``superconvergence'' phenomenon of the method still awaits explanation.

These observations were also confirmed by the numerical results for other $\alpha$ values; see Fig. \ref{fig:alcap}
for the convergence behavior for four different $\alpha$ values, in case of the discontinuous potential $q_3$.
As the $\alpha$ value increases towards two, the number of real eigenvalues (which appear always in pair) also
increases accordingly. The overall convergence seems relatively independent of the $\alpha$ values,
except for $\alpha=7/4$, for which there are four real eigenvalues. Here the convergence rates for the first
and second eigenvalues are different, with one slightly below two and the other slightly above two; see
also Table \ref{tab:al75capq3}. A similar behavior can be observed for the third and fourth eigenvalues, although
the difference is less dramatic. The rest of the eigenvalues exhibits a second-order convergence, consistent
with other cases. Hence the abnormality is only observed for the ``newly'' emerged real eigenvalues.
The cause of the abnormal convergence behavior in the transient region is still unclear.

\begin{table}[h!]
\centering
\caption{The errors of the first ten eigenvalues for $\alpha=5/3$, $q_1$, Caputo derivative,
computed on a uniform mesh of mesh size $1/(10\times2^k)$, $k=3,\ldots,8$. The first two
eigenvalues are real and the rest consists of complex conjugates.}\label{tab:al66capq1}
\begin{tabular}{ccccccccc}
\hline
  $e\backslash k$ & 3 & 4 & 5 & 6 & 7 & 8 & rate\\
\hline
  $\lambda_1$     & 1.03e-3 & 2.55e-4 & 6.33e-5 & 1.57e-5 & 3.89e-6 & 9.37e-7 & 2.01\\
  $\lambda_2$     & 1.64e-3 & 3.33e-4 & 6.80e-5 & 1.39e-5 & 2.87e-6 & 6.02e-7 & 2.28\\
  $\lambda_{3,4}$ & 3.75e-2 & 8.19e-3 & 1.82e-3 & 4.11e-4 & 9.35e-5 & 2.07e-5 & 2.15\\
  $\lambda_{5,6}$ & 1.58e-1 & 3.32e-2 & 7.08e-3 & 1.53e-3 & 3.34e-4 & 7.18e-5 & 2.20\\
  $\lambda_{7,8}$ & 4.87e-1 & 1.00e-1 & 2.09e-2 & 4.40e-3 & 9.35e-4 & 1.95e-4 & 2.24\\
  $\lambda_{9,10}$& 1.18e0  & 2.39e-1 & 4.92e-2 & 1.02e-2 & 2.13e-3 & 4.37e-4 & 2.27\\
\hline
\end{tabular}
\end{table}

\begin{table}[h!]
\centering
\caption{The errors of the first 10 eigenvalues for $\alpha=5/3$, $q_2$, Caputo derivative,
computed on a uniform mesh of mesh size $1/(10\times2^k)$, $k=3,\ldots,8$. The first two
eigenvalues are real and the rest consists of complex conjugates.}\label{tab:al66capq2}
\begin{tabular}{ccccccccc}
\hline
  $e\backslash k$ & 3 & 4 & 5 & 6 & 7 & 8 & rate\\
\hline
  $\lambda_1$     & 1.11e-3 & 2.75e-4 & 6.79e-5 & 1.68e-5 & 4.14e-6 & 9.95e-7 & 2.02\\
  $\lambda_2$     & 1.75e-3 & 3.59e-4 & 7.38e-5 & 1.53e-5 & 3.18e-6 & 6.74e-7 & 2.27\\
  $\lambda_{3,4}$ & 3.75e-2 & 8.21e-3 & 1.83e-3 & 4.13e-4 & 9.39e-5 & 2.08e-5 & 2.15\\
  $\lambda_{5,6}$ & 1.58e-1 & 3.33e-2 & 7.09e-3 & 1.53e-3 & 3.35e-4 & 7.20e-5 & 2.21\\
  $\lambda_{7,8}$ & 4.87e-1 & 1.00e-1 & 2.09e-2 & 4.40e-3 & 9.35e-4 & 1.96e-4 & 2.24\\
  $\lambda_{9,10}$& 1.18e0  & 2.39e-1 & 4.92e-2 & 1.02e-2 & 2.13e-3 & 4.37e-4 & 2.27\\
\hline
\end{tabular}
\end{table}

\begin{table}[h!]
\centering
\caption{The error of the first ten eigenvalues for $\alpha=5/3$, $q_3$, Caputo derivative,
computed on a uniform mesh of mesh size $1/(10\times2^k)$, $k=3,\ldots,8$. The first two
eigenvalues are real and the rest consists of complex conjugates.}\label{tab:al66capq3}
\begin{tabular}{ccccccccc}
\hline
  $e\backslash k$ & 3 & 4 & 5 & 6 & 7 & 8 & rate\\
\hline
  $\lambda_1$     & 1.10e-3 & 2.71e-4 & 6.69e-5 & 1.66e-5 & 4.08e-6 & 9.81e-7 & 2.02\\
  $\lambda_2$     & 1.73e-3 & 3.53e-4 & 7.23e-5 & 1.48e-5 & 3.08e-6 & 6.47e-7 & 2.27\\
  $\lambda_{3,4}$ & 3.79e-2 & 8.31e-3 & 1.85e-3 & 4.18e-4 & 9.52e-5 & 2.11e-5 & 2.14\\
  $\lambda_{5,6}$ & 1.58e-1 & 3.33e-2 & 7.09e-3 & 1.53e-3 & 3.35e-4 & 7.20e-5 & 2.20\\
  $\lambda_{7,8}$ & 4.87e-1 & 1.00e-1 & 2.08e-2 & 4.40e-3 & 9.34e-4 & 1.95e-4 & 2.24\\
  $\lambda_{9,10}$& 1.18e0  & 2.39e-1 & 4.92e-2 & 1.02e-2 & 2.13e-3 & 4.37e-4 & 2.27\\
\hline
\end{tabular}
\end{table}

\begin{figure}[h!]
  \centering
  \begin{tabular}{cc}
  \includegraphics[trim = .5cm 0cm 1cm 0cm, clip=true,width=8cm]{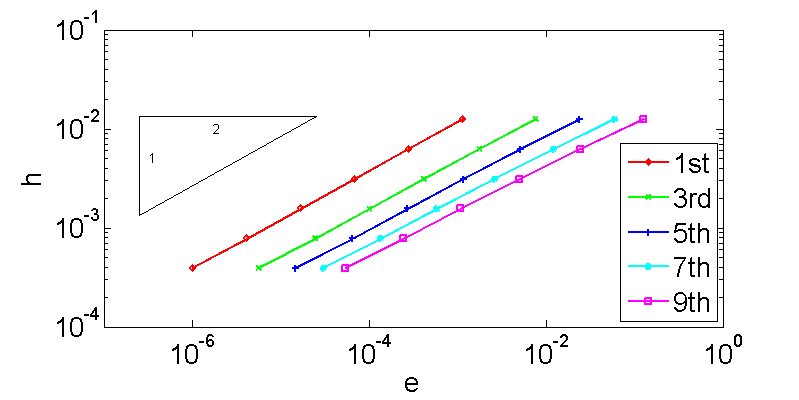}
  &\includegraphics[trim = .5cm 0cm 1cm 0cm, clip=true,width=8cm]{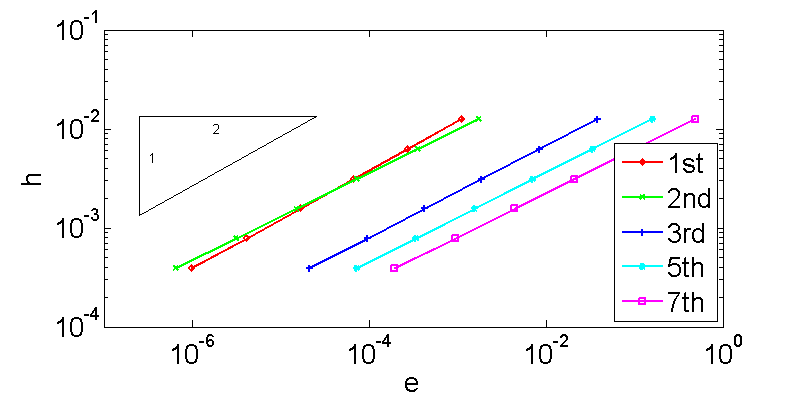}\\
    $\alpha = 4/3$ & $\alpha=5/3$\\
    \includegraphics[trim = .5cm 0cm 1cm 0cm, clip=true,width=8cm]{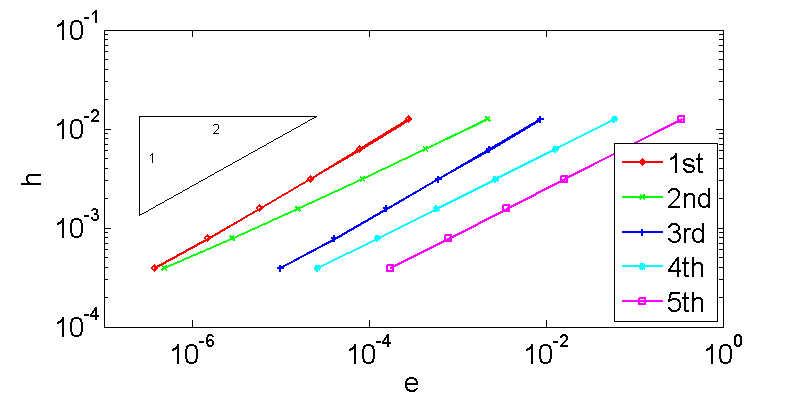}
  &\includegraphics[trim = .5cm 0cm 1cm 0cm, clip=true,width=8cm]{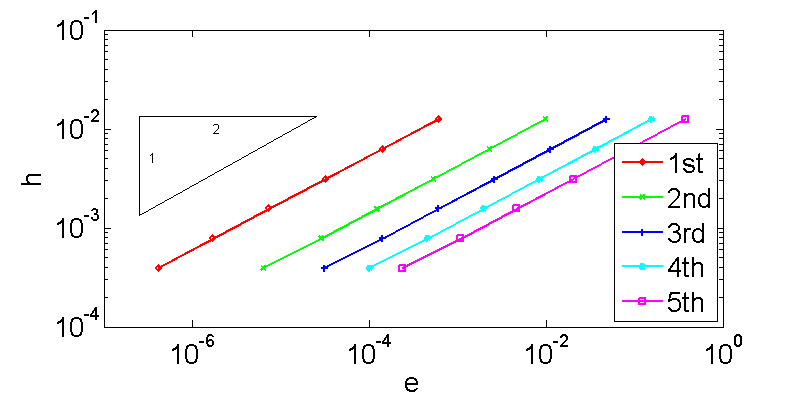}\\
  $\alpha= 7/4$ & $\alpha = 19/10$
  \end{tabular}
  \caption{The convergence of the finite element approximations of the eigenvalues with the potential $q_3$,
    for Caputo derivative with $\alpha=4/3$, $5/3$, $7/4$, and $19/10$. }\label{fig:alcap}
\end{figure}

\begin{table}[h!]
\centering
\caption{The errors of the first ten eigenvalues for $\alpha=7/4$, $q_3$, Caputo derivative,
computed on a uniform mesh of mesh size $1/(10\times2^k)$, $k=3,\ldots,8$. The first four
eigenvalues are real and the rest consists of complex conjugates.}\label{tab:al75capq3}
\begin{tabular}{ccccccccc}
\hline
  $e\backslash k$ & 3 & 4 & 5 & 6 & 7 & 8 & rate\\
  \hline
    $\lambda_1$     & 2.79e-4 & 7.84e-5 & 2.14e-5 & 5.73e-6 & 1.50e-6 & 3.74e-7 & 1.92\\
    $\lambda_2$     & 2.19e-3 & 4.30e-4 & 8.27e-5 & 1.54e-5 & 2.77e-6 & 4.79e-7 & 2.40\\
    $\lambda_3$     & 8.63e-3 & 2.29e-3 & 6.05e-4 & 1.58e-4 & 4.06e-5 & 9.85e-6 & 1.95\\
    $\lambda_4$     & 5.89e-2 & 1.26e-2 & 2.71e-3 & 5.80e-4 & 1.24e-4 & 2.57e-5 & 2.22\\
    $\lambda_{5,6}$ & 3.44e-1 & 7.44e-2 & 1.63e-2 & 3.60e-3 & 7.96e-4 & 1.72e-4 & 2.18\\
    $\lambda_{7,8}$ & 8.46e-1 & 1.81e-1 & 3.89e-2 & 8.41e-3 & 1.81e-3 & 3.82e-4 & 2.21\\
    $\lambda_{9,10}$& 2.06e0  & 4.38e-1 & 9.35e-2 & 2.00e-2 & 4.28e-3 & 8.93e-4 & 2.22\\
  \hline
\end{tabular}
\end{table}

As was mentioned above, the Caputo eigenfunctions are fairly regular. We illustrate this in
Fig. \ref{fig:eigfcn:cap}, where the first, fifth and tenth eigenfunctions for three different
$\alpha$ values, i.e., $4/3$, $5/3$ and $19/10$, were shown. The eigenfunctions are normalized
to have unit $L^2\II$-norm. The profiles of the eigenfunctions generally resemble sinusoidal
functions, but the magnitudes are attenuated around $x=1$, which is related to the asymptotics
of the Mittag-Leffler function $xE_{\alpha,2}(-\lambda_n x^\alpha)$, the eigenfunction in the
case of a zero potential. The degree of attenuation depends
crucially on the order $\alpha$ of the fractional derivative. Further, as the eigenvalue increases,
the corresponding eigenfunctions are increasingly oscillatory, and the number of interior
zero differs by one for the real and imaginary parts of a genuinely complex eigenfunction.

\begin{figure}[htp!]
  \centering
  \begin{tabular}{cc}
    \includegraphics[trim = 1cm 0cm 2cm 0cm, clip=true,width=6cm]{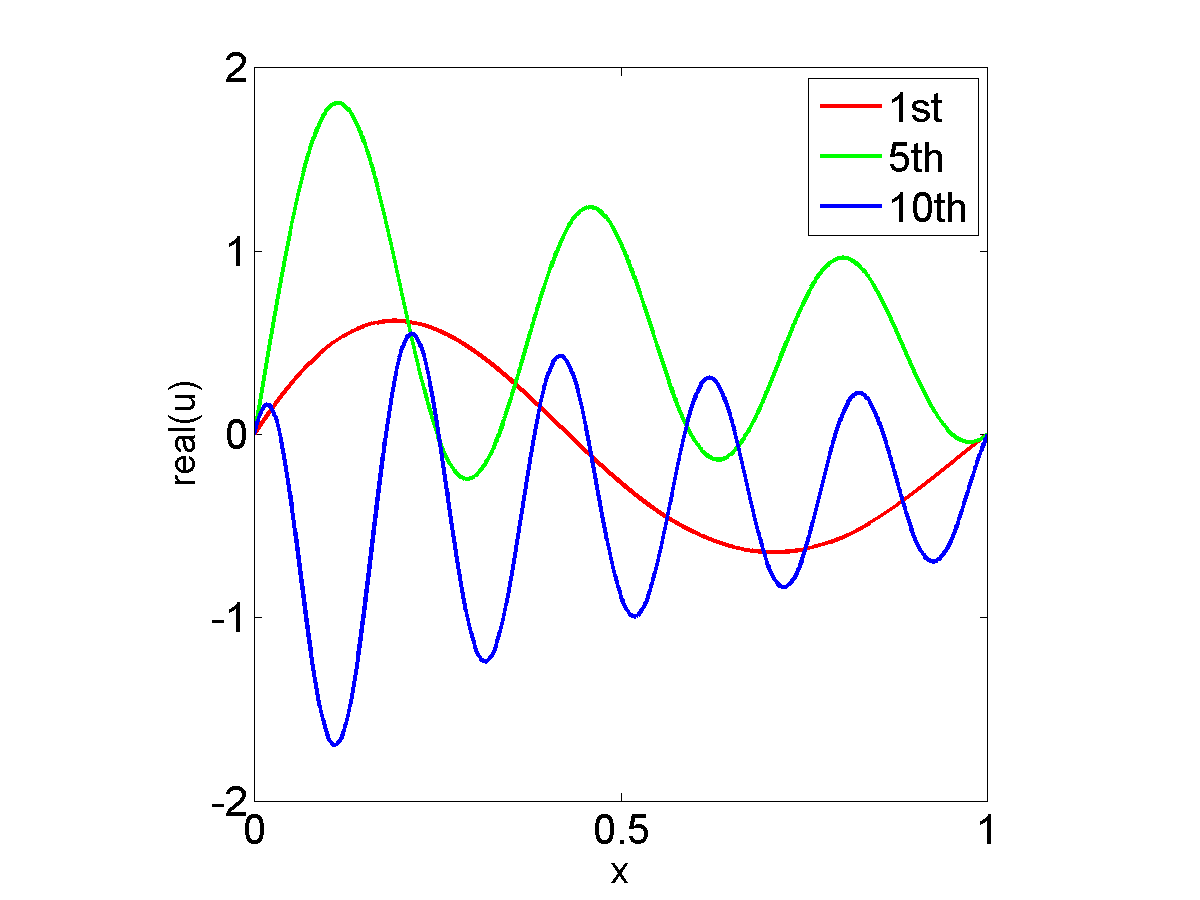} &
    \includegraphics[trim = 1cm 0cm 2cm 0cm, clip=true,width=6cm]{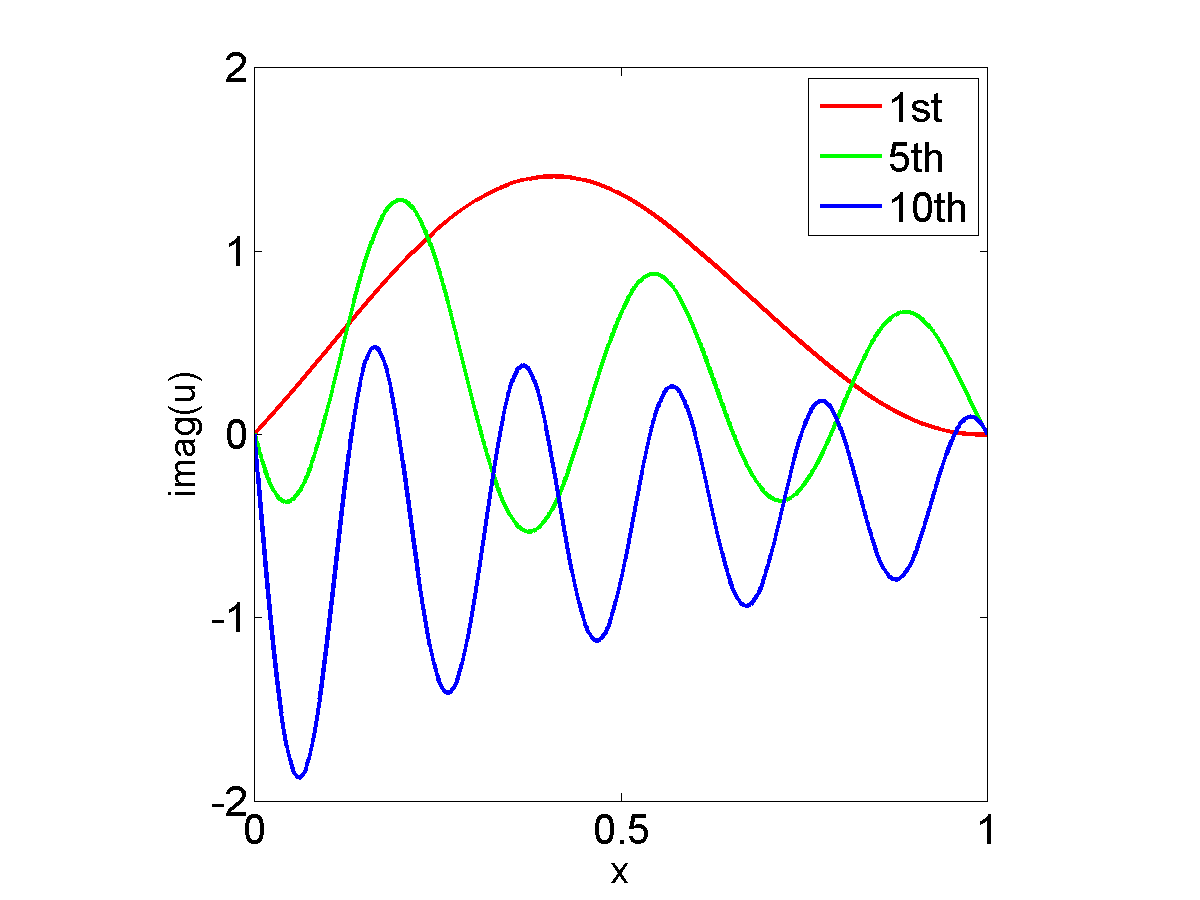}\\
    \includegraphics[trim = 1cm 0cm 2cm 0cm, clip=true,width=6cm]{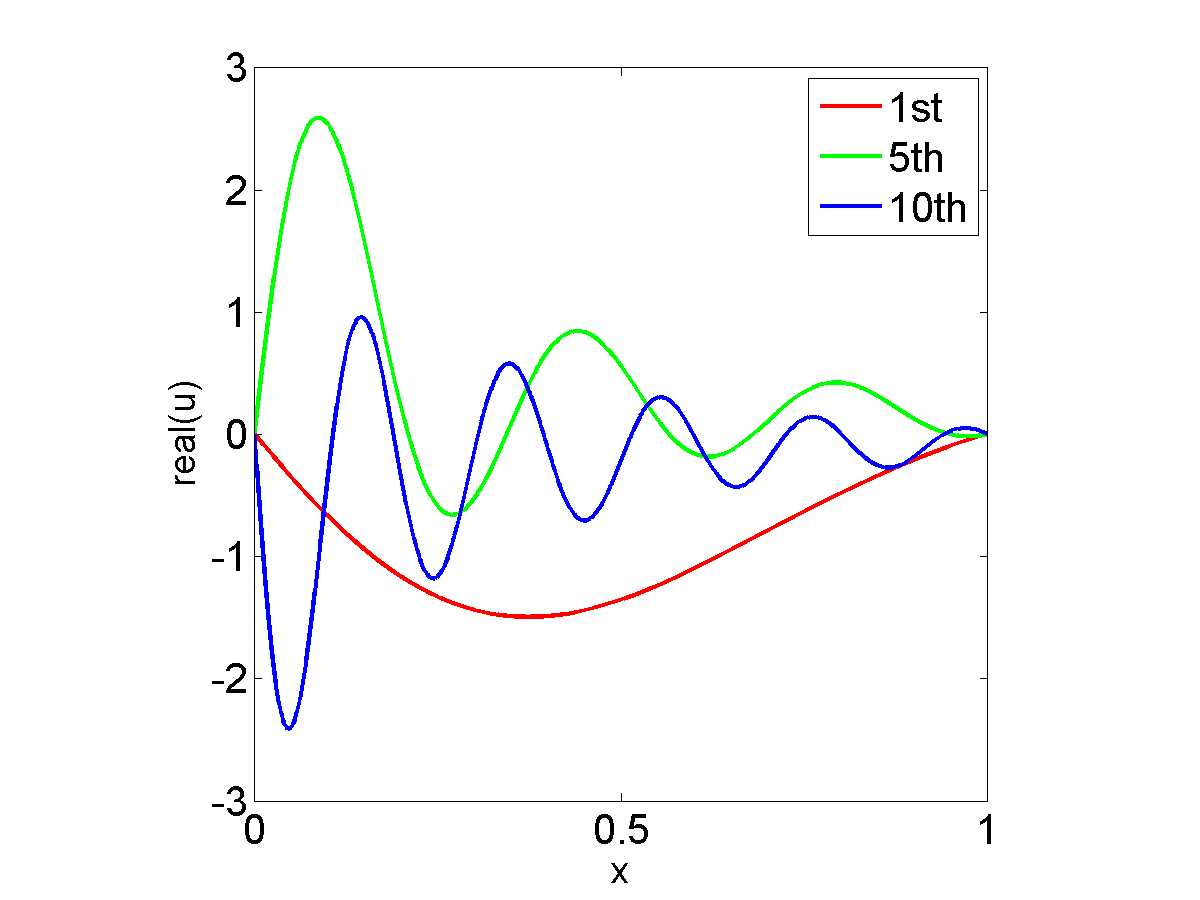} &
    \includegraphics[trim = 1cm 0cm 2cm 0cm, clip=true,width=6cm]{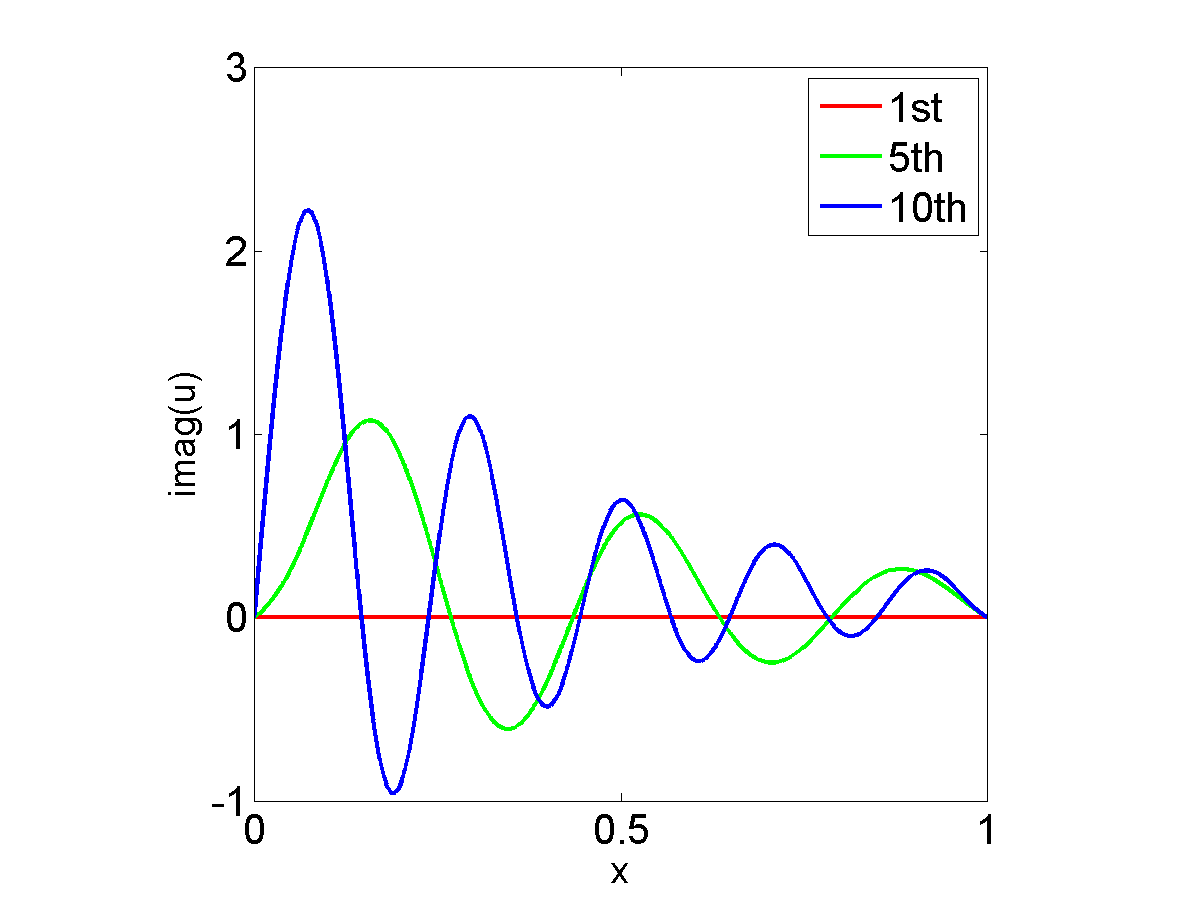}\\
    \includegraphics[trim = 1cm 0cm 2cm 0cm, clip=true,width=6cm]{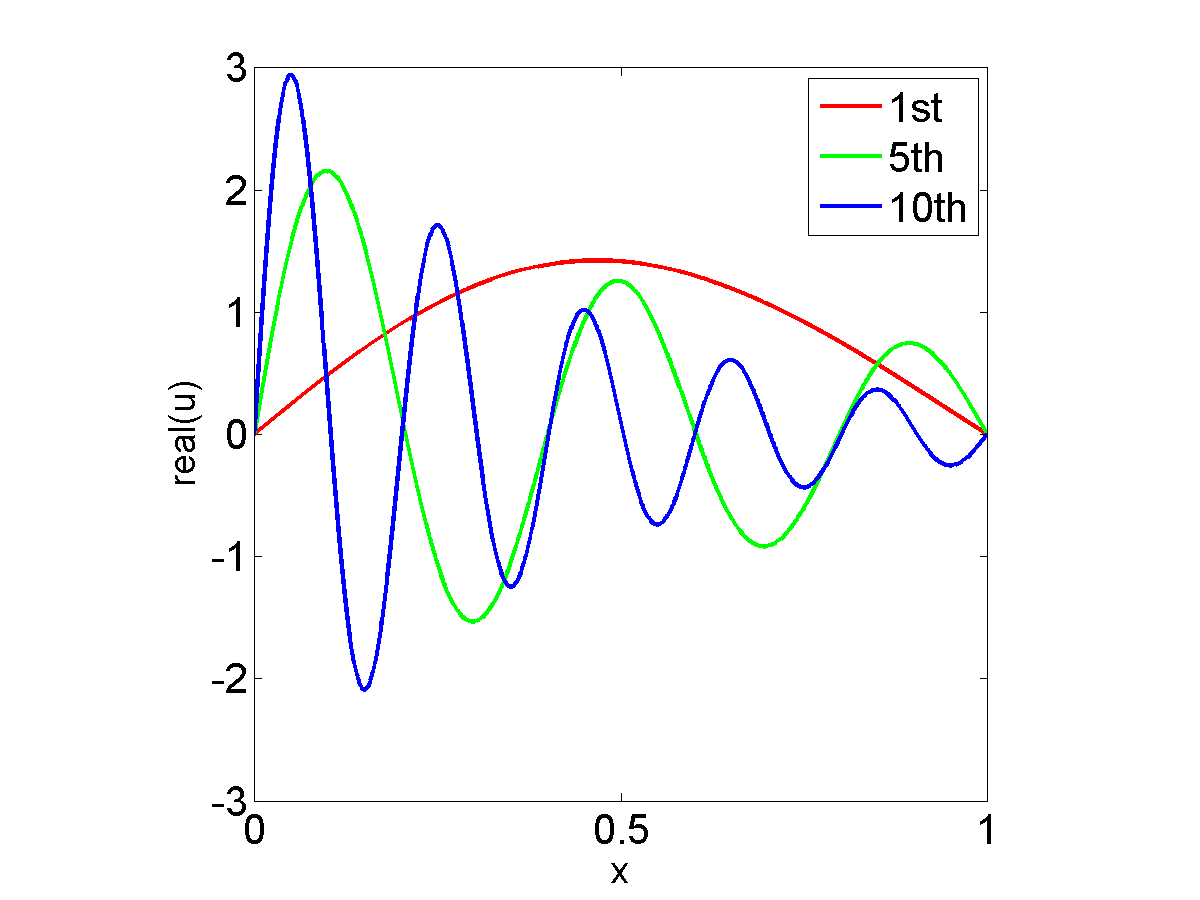} &
    \includegraphics[trim = 1cm 0cm 2cm 0cm, clip=true,width=6cm]{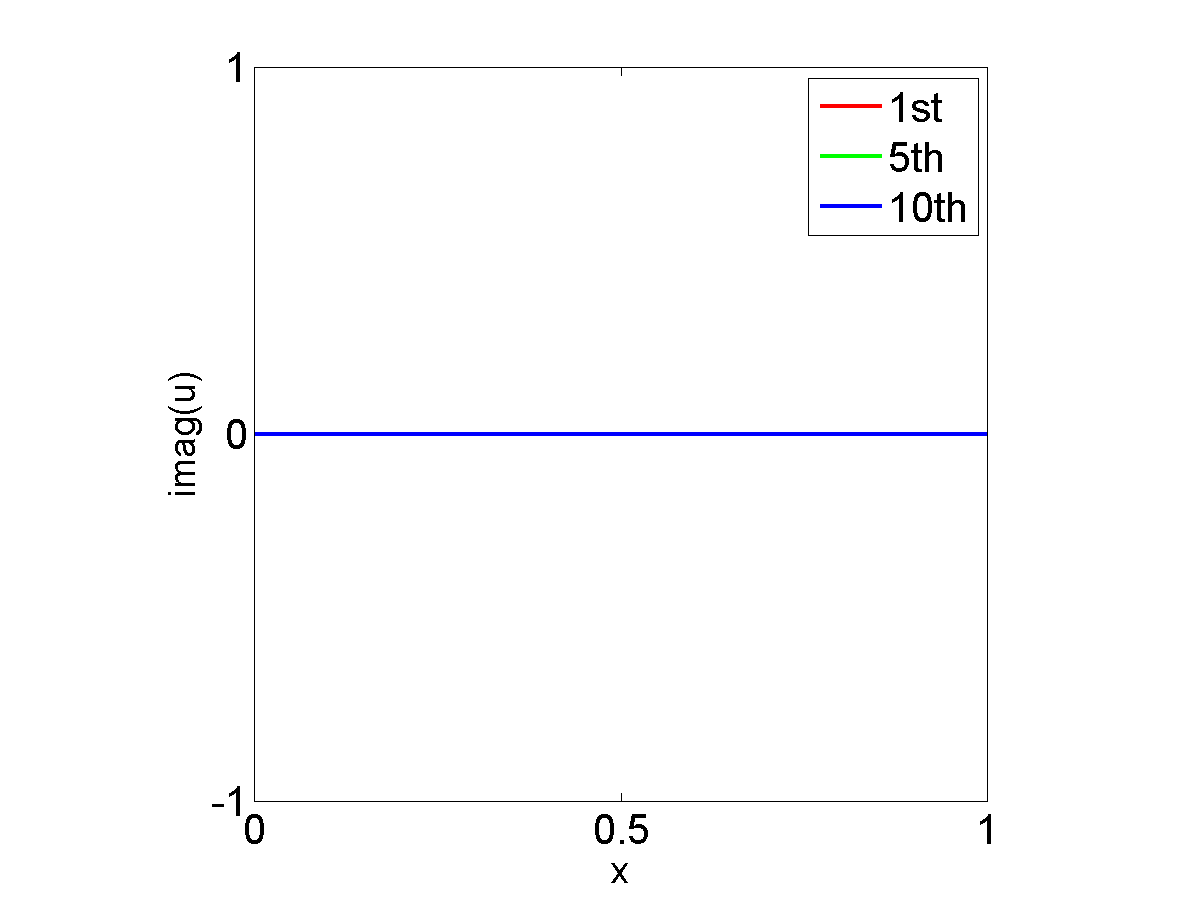}\\
     real part & imaginary part
  \end{tabular}
  \caption{The profile of the first, fifth and tenth eigenfunctions in case of $q_3$, $\alpha=4/3$
  (top row), $\alpha=5/3$ (middle row) and $\alpha=19/10$ (bottom row), Caputo case.}\label{fig:eigfcn:cap}
\end{figure}

\subsection{Riemann-Liouville derivative case}
Now we turn to the numerically more challenging case of Riemann-Liouville derivative. By the regularity theory
in Section \ref{ssec:reg}, the eigenfunctions are less regular, with an inherent singularity concentrated at
the origin and the degree of singularity is precisely of order $x^{\alpha-1}$. Hence, one would naturally expect
a slow convergence of the finite element approximations with a uniform mesh. Our experiments indicate that
finite element approximations of the eigenfunctions for $\alpha$ close to unity indeed suffer from pronounced
oscillations near the origin. The finite element method does converge, and hence the oscillations will go away
as the mesh refines. Nonetheless, the oscillations still cast doubt into the accuracy of the eigenvalue approximations.

The numerical results for $\alpha=5/3$ with the three potentials are presented in Tables
\ref{tab:al66riemq1}-\ref{tab:al66riemq3}. For all three potentials, the first nine eigenvalues are
real and the rest appears as conjugate pairs, and in the tables, we show the convergence results only
for the first six eigenvalues. Surprisingly, the eigenvalue approximations exhibit
consistently a second-order convergence, identical with that for the Caputo derivative.
Like before, the convergence rate is almost independent of the presence of a potential term.
The convergence behavior is further illustrated in Fig. \ref{fig:alriem} for different $\alpha$ values
in the case of the discontinuous potential $q_3$. The preceding observation remains largely true,
except within a ``transient'' region $1.70\leq \alpha\leq 1.85$ to which the value $\alpha=7/4$
belongs: the method still converges almost at the same rate, but the convergence is not as steady
as for other cases. However, although not presented, we would like to remark that the convergence
for larger eigenvalues is rather steady. We also emphasize that the good convergence of the
eigenvalue approximations, especially for $\alpha$ close to one, has not been theoretically backed up.

\begin{table}[h!]
\centering
\caption{The errors of the first six eigenvalues, which are all real, for $\alpha=5/3$, $q_1$,
Riemann-Liouville derivative, computed on a uniform mesh of mesh size $1/(10\times2^k)$,
$k=3,\ldots,8$.}\label{tab:al66riemq1}
\begin{tabular}{ccccccccc}
\hline
  $e\backslash k$ & 3 & 4 & 5 & 6 & 7 & 8 & rate\\
\hline
  $\lambda_1$ & 3.53e-4 & 9.36e-5 & 2.44e-5 & 6.31e-6 & 1.60e-6 & 3.95e-7 & 1.95\\
  $\lambda_2$ & 2.73e-3 & 7.30e-4 & 1.92e-4 & 4.98e-5 & 1.27e-5 & 3.09e-6 & 1.96\\
  $\lambda_3$ & 7.33e-3 & 1.99e-3 & 5.32e-4 & 1.39e-4 & 3.58e-5 & 8.64e-6 & 1.96\\
  $\lambda_4$ & 1.81e-2 & 4.80e-3 & 1.27e-3 & 3.31e-4 & 8.46e-5 & 2.05e-5 & 1.96\\
  $\lambda_5$ & 2.62e-2 & 7.11e-3 & 1.93e-3 & 5.16e-4 & 1.35e-4 & 3.39e-5 & 1.95\\
  $\lambda_6$ & 6.59e-2 & 1.65e-2 & 4.24e-3 & 1.09e-3 & 2.77e-4 & 6.81e-5 & 1.98\\
\hline
\end{tabular}
\end{table}

\begin{table}[h!]
\centering
\caption{The errors of the first six eigenvalues, which are all real, for $\alpha=5/3$, $q_2$,
Riemann-Liouville derivative, computed on a uniform mesh of mesh size $1/(10\times2^k)$,
$k=3,\ldots,8$.}\label{tab:al66riemq2}
\begin{tabular}{ccccccccc}
\hline
  $e\backslash k$ & 3 & 4 & 5 & 6 & 7 & 8 & rate\\
\hline
  $\lambda_1$ & 3.68e-4 & 9.78e-5 & 2.56e-5 & 6.61e-6 & 1.68e-6 & 4.13e-7 & 1.96\\
  $\lambda_2$ & 2.78e-3 & 7.43e-4 & 1.96e-4 & 5.08e-5 & 1.29e-5 & 3.14e-6 & 1.96\\
  $\lambda_3$ & 7.41e-3 & 2.01e-3 & 5.38e-4 & 1.41e-4 & 3.62e-5 & 8.75e-6 & 1.95\\
  $\lambda_4$ & 1.82e-2 & 4.83e-3 & 1.27e-3 & 3.33e-4 & 8.52e-5 & 2.06e-5 & 1.95\\
  $\lambda_5$ & 2.63e-2 & 7.15e-3 & 1.94e-3 & 5.20e-4 & 1.36e-4 & 3.41e-5 & 1.94\\
  $\lambda_6$ & 6.60e-2 & 1.65e-2 & 4.25e-3 & 1.09e-3 & 2.78e-4 & 6.83e-5 & 1.97\\
\hline
\end{tabular}
\end{table}

\begin{table}[h!]
\centering
\caption{The errors of the first six eigenvalues, which are all real, for $\alpha=5/3$, $q_3$,
Riemann-Liouville derivative, computed on a uniform mesh of a mesh size $1/(10\times2^k)$,
$k=3,\ldots,8$.}\label{tab:al66riemq3}
\begin{tabular}{ccccccccc}
\hline
  $e\backslash k$ & 3 & 4 & 5 & 6 & 7 & 8 & rate\\
\hline
  $\lambda_1$ & 3.60e-4 & 9.57e-5 & 2.50e-5 & 6.46e-6 & 1.64e-6 & 4.04e-7 & 1.96\\
  $\lambda_2$ & 2.78e-3 & 7.44e-4 & 1.96e-4 & 5.08e-5 & 1.29e-5 & 3.14e-6 & 1.95\\
  $\lambda_3$ & 7.35e-3 & 1.99e-3 & 5.34e-4 & 1.40e-4 & 3.59e-5 & 8.68e-6 & 1.95\\
  $\lambda_4$ & 1.80e-2 & 4.79e-3 & 1.27e-3 & 3.31e-4 & 8.45e-5 & 2.05e-5 & 1.95\\
  $\lambda_5$ & 2.63e-2 & 7.13e-3 & 1.94e-3 & 5.18e-4 & 1.35e-4 & 3.40e-5 & 1.94\\
  $\lambda_6$ & 6.59e-2 & 1.65e-2 & 4.24e-3 & 1.09e-3 & 2.77e-4 & 6.81e-5 & 1.97\\
\hline
\end{tabular}
\end{table}

\begin{figure}[h!]
  \centering
  \begin{tabular}{cc}
   \includegraphics[trim = .5cm 0cm 1cm 0cm, clip=true,width=8cm]{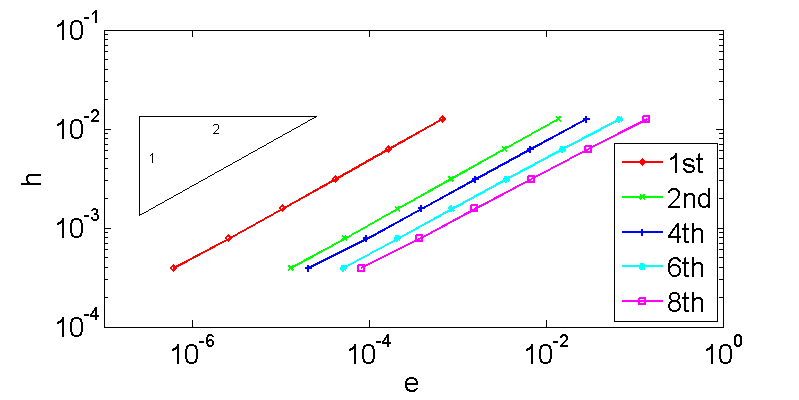}
  &\includegraphics[trim = .5cm 0cm 1cm 0cm, clip=true,width=8cm]{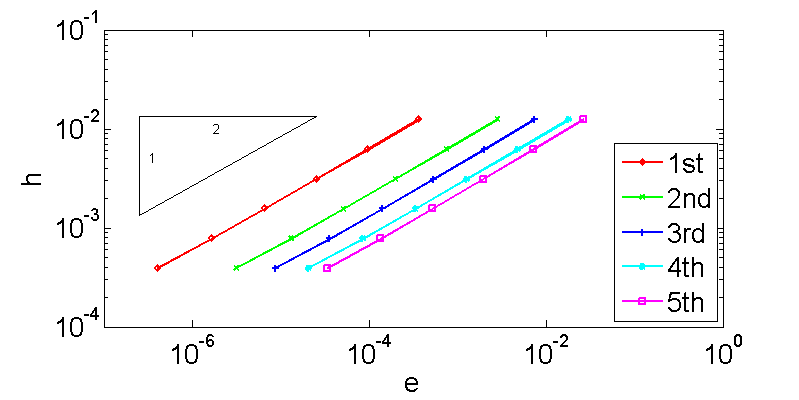}\\
    $\alpha = 4/3$ & $\alpha=5/3$\\
   \includegraphics[trim = .5cm 0cm 1cm 0cm, clip=true,width=8cm]{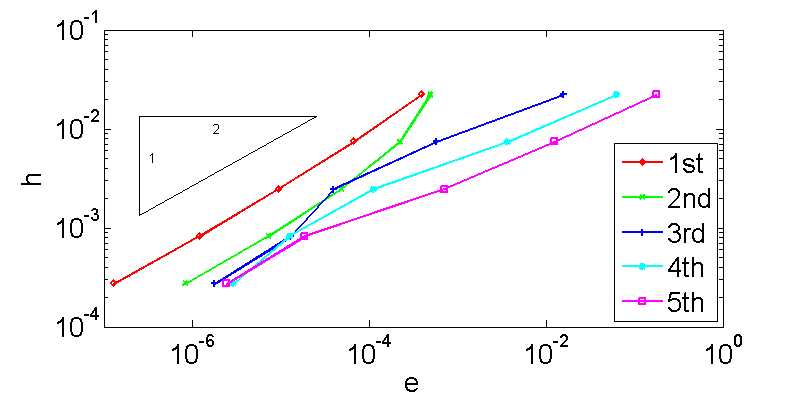}
  &\includegraphics[trim = .5cm 0cm 1cm 0cm, clip=true,width=8cm]{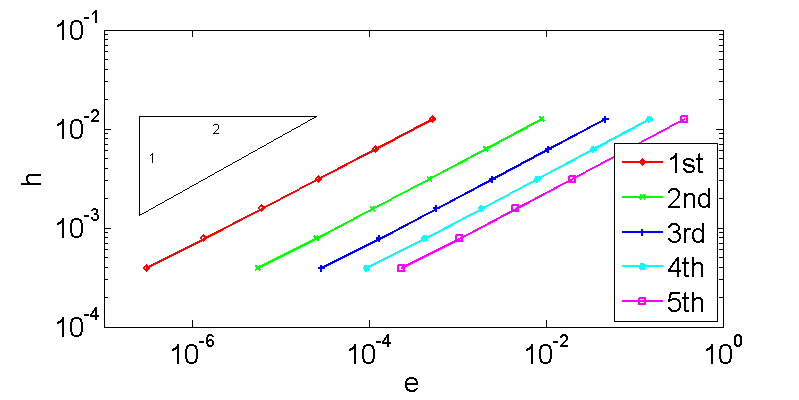}\\
  $\alpha= 7/4$ & $\alpha = 19/10$
  \end{tabular}
  \caption{The convergence of the finite element approximations of the eigenvalues with the potential $q_3$,
    for Riemann-Liouville derivative with $\alpha=4/3$, $5/3$, $7/4$, and $19/10$. }\label{fig:alriem}
\end{figure}

Our theory predicts that the eigenfunctions for the Caputo and Riemann-Liouville derivative
behave very differently. To confirm this, we plot eigenfunctions for the latter in Fig.
\ref{fig:eigfcn:riem}. One can observe from Figs. \ref{fig:eigfcn:cap} and \ref{fig:eigfcn:riem} that
apart from a stronger singularity at the origin, the Riemann-Liouville eigenfunctions are also
far more significantly attenuated towards $x=1$. In case of $q=0$, this
might be explained by the exponential asymtotics of the Mittag-Leffler function: the eigenfunctions
in the Caputo and Riemann-Liouville cases are given by $xE_{\alpha,2}(-\lambda_n x^\alpha)$ and
$x^{\alpha-1}E_{\alpha,\alpha}(-\lambda_n x^\alpha)$, respectively; the former decays only linearly,
whereas the latter decays quadratically \cite[pp. 43]{KilbasSrivastavaTrujillo:2006}; This can also
be deduced from Fig. \ref{fig:mitlef}: the wells, which correspond to zeros of the Mittag-Leffler functions,
run much deeper for Riemann-Liouville case. The difference in the decay behavior may reflect
completely different physics behind the two fractional derivatives.

\begin{figure}[htp!]
  \centering
  \begin{tabular}{cc}
    \includegraphics[trim = 1cm 0cm 2cm 0cm, clip=true,width=6cm]{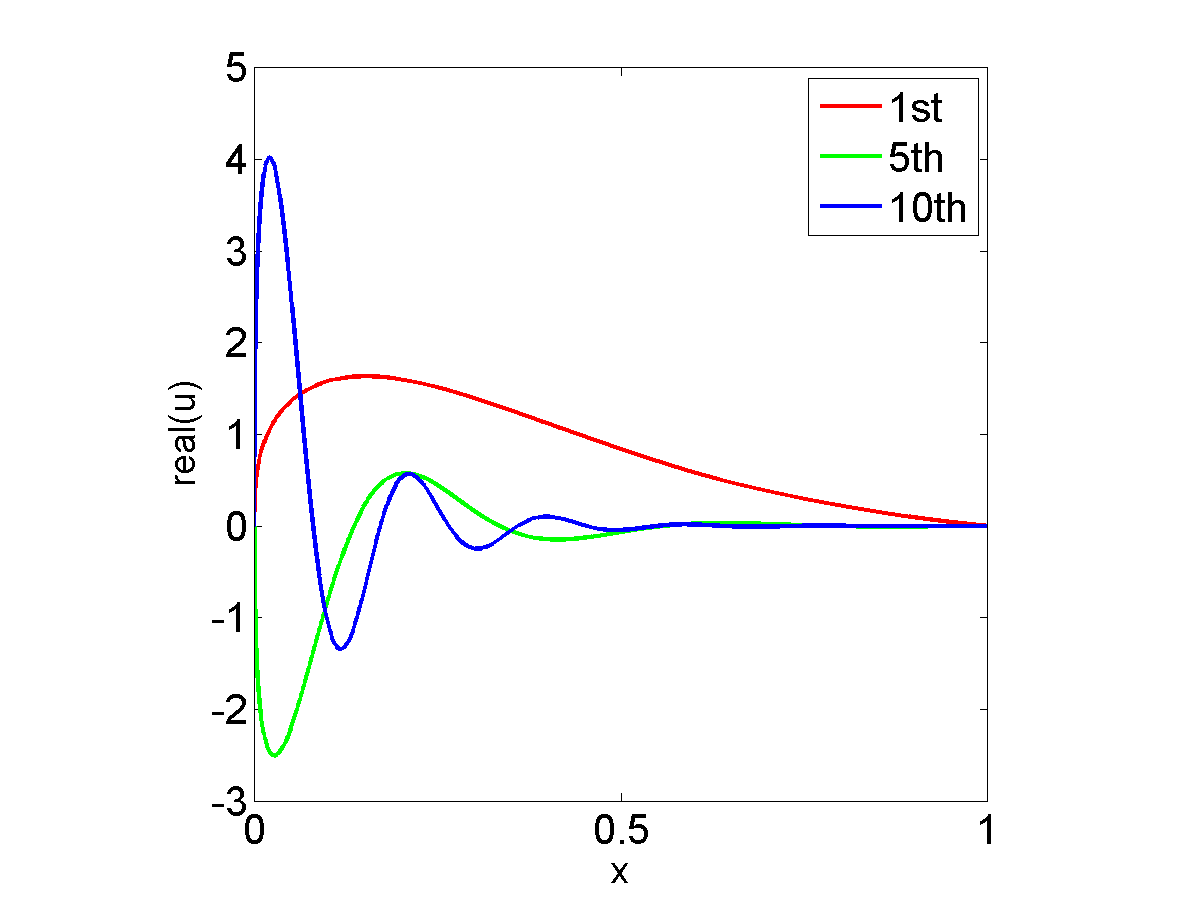} &
    \includegraphics[trim = 1cm 0cm 2cm 0cm, clip=true,width=6cm]{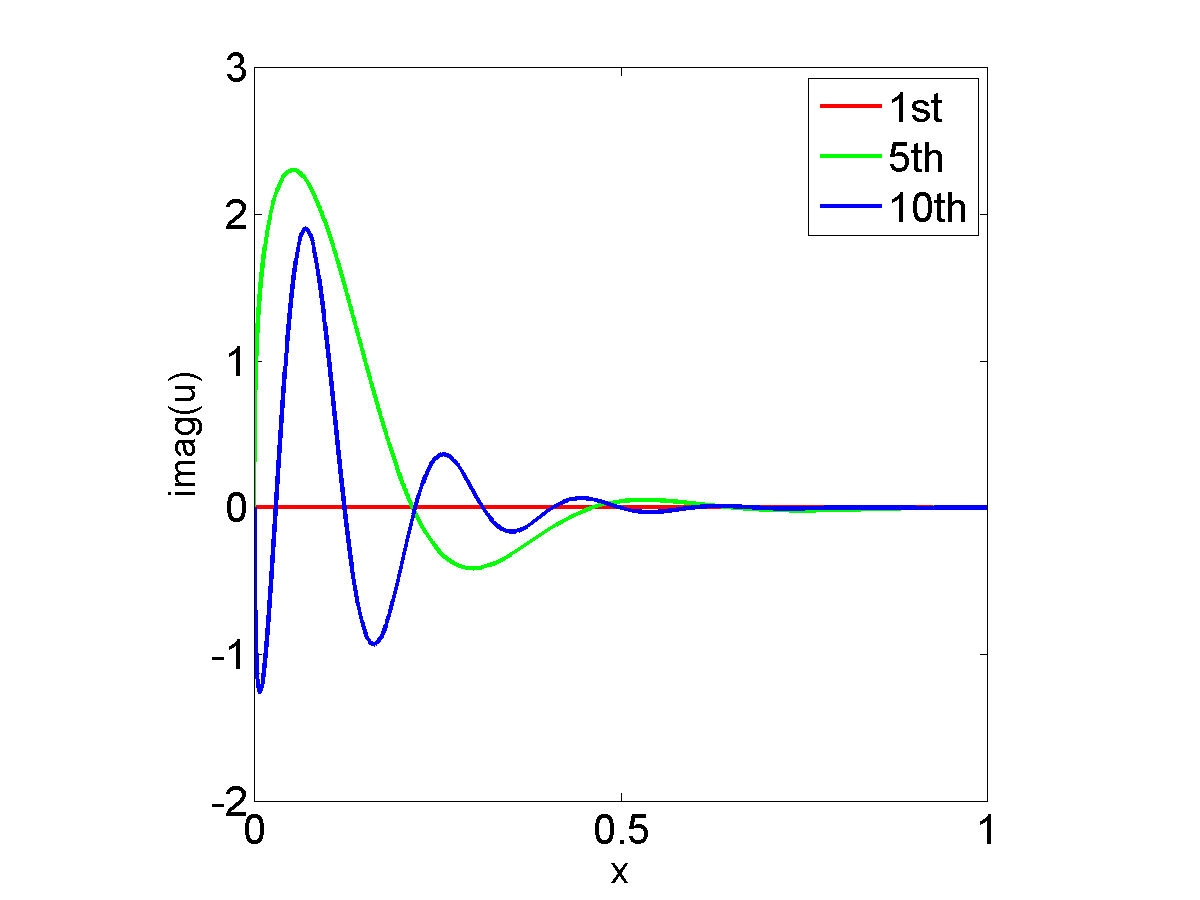}\\
    \includegraphics[trim = 1cm 0cm 2cm 0cm, clip=true,width=6cm]{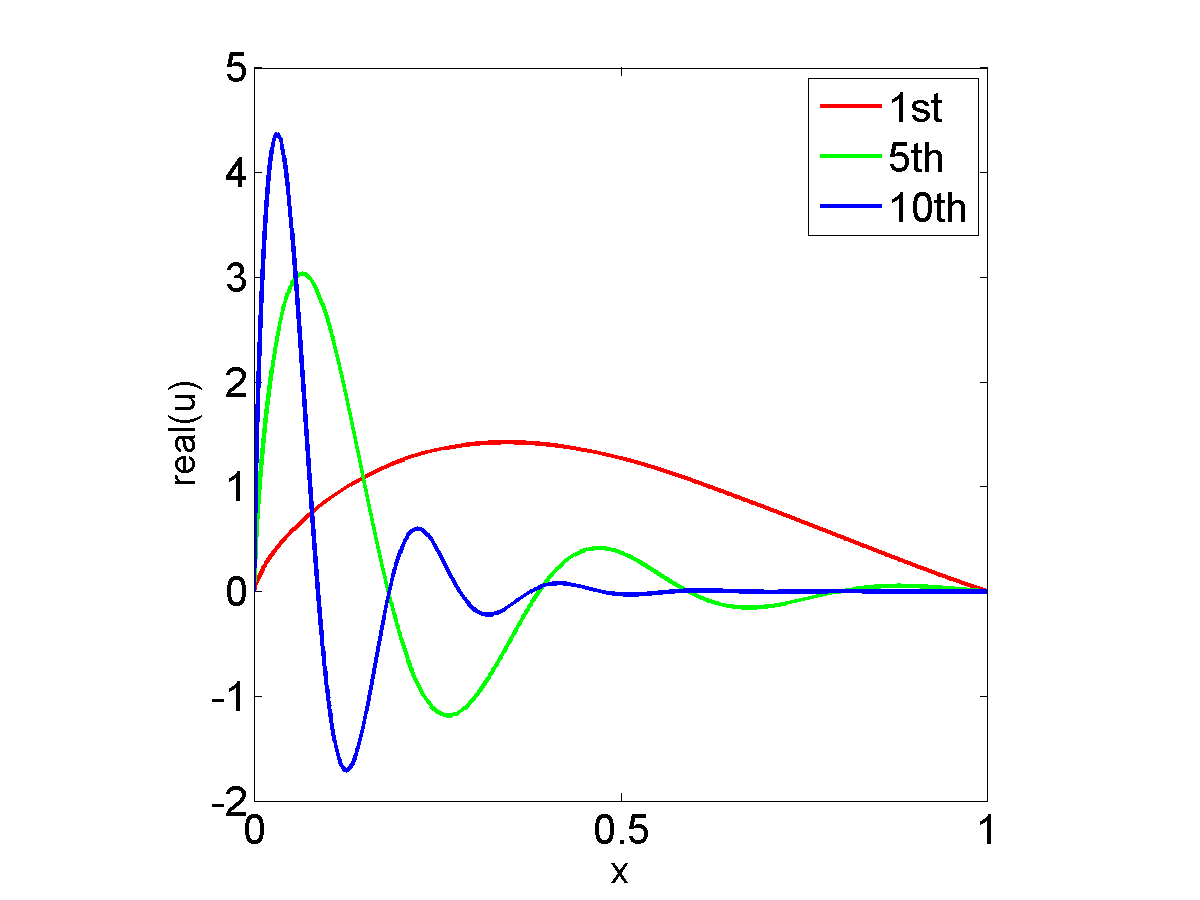} &
    \includegraphics[trim = 1cm 0cm 2cm 0cm, clip=true,width=6cm]{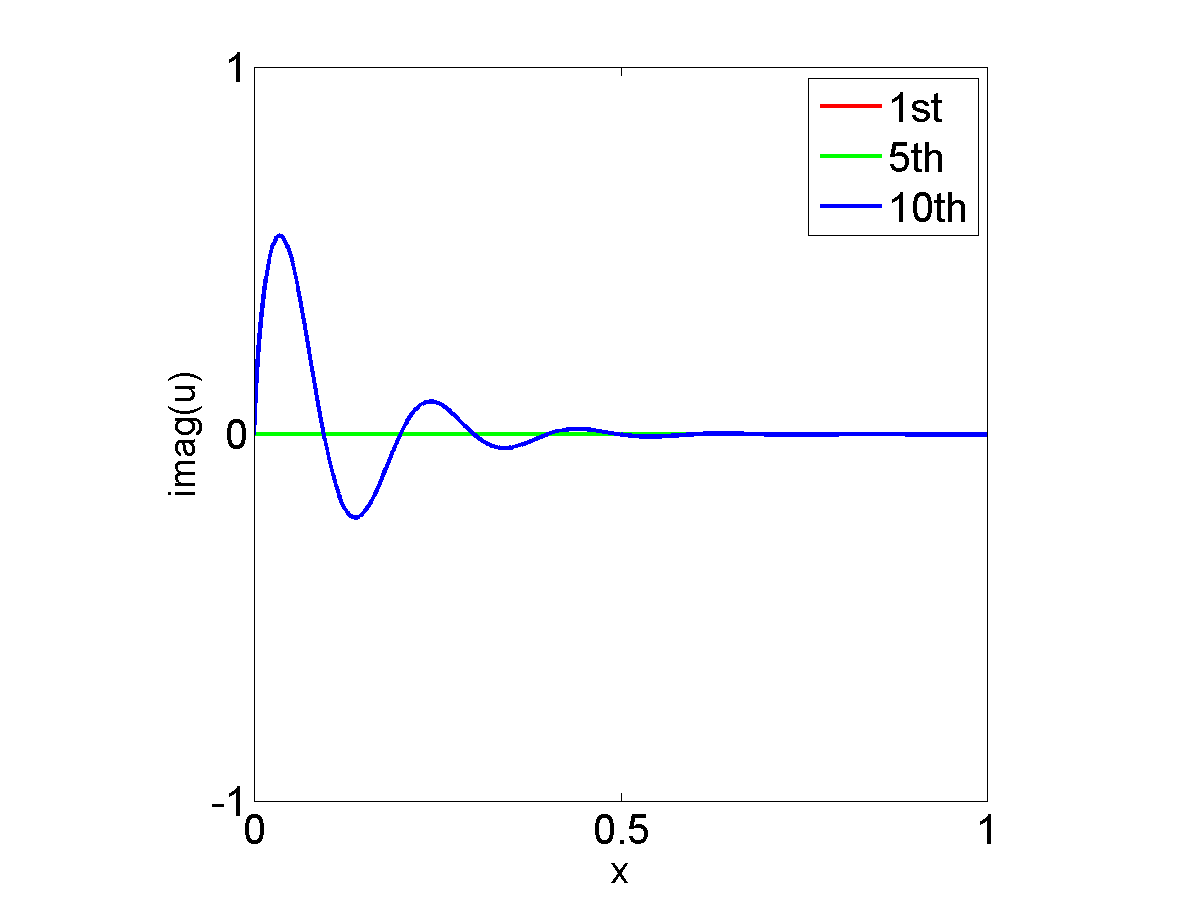}\\
    \includegraphics[trim = 1cm 0cm 2cm 0cm, clip=true,width=6cm]{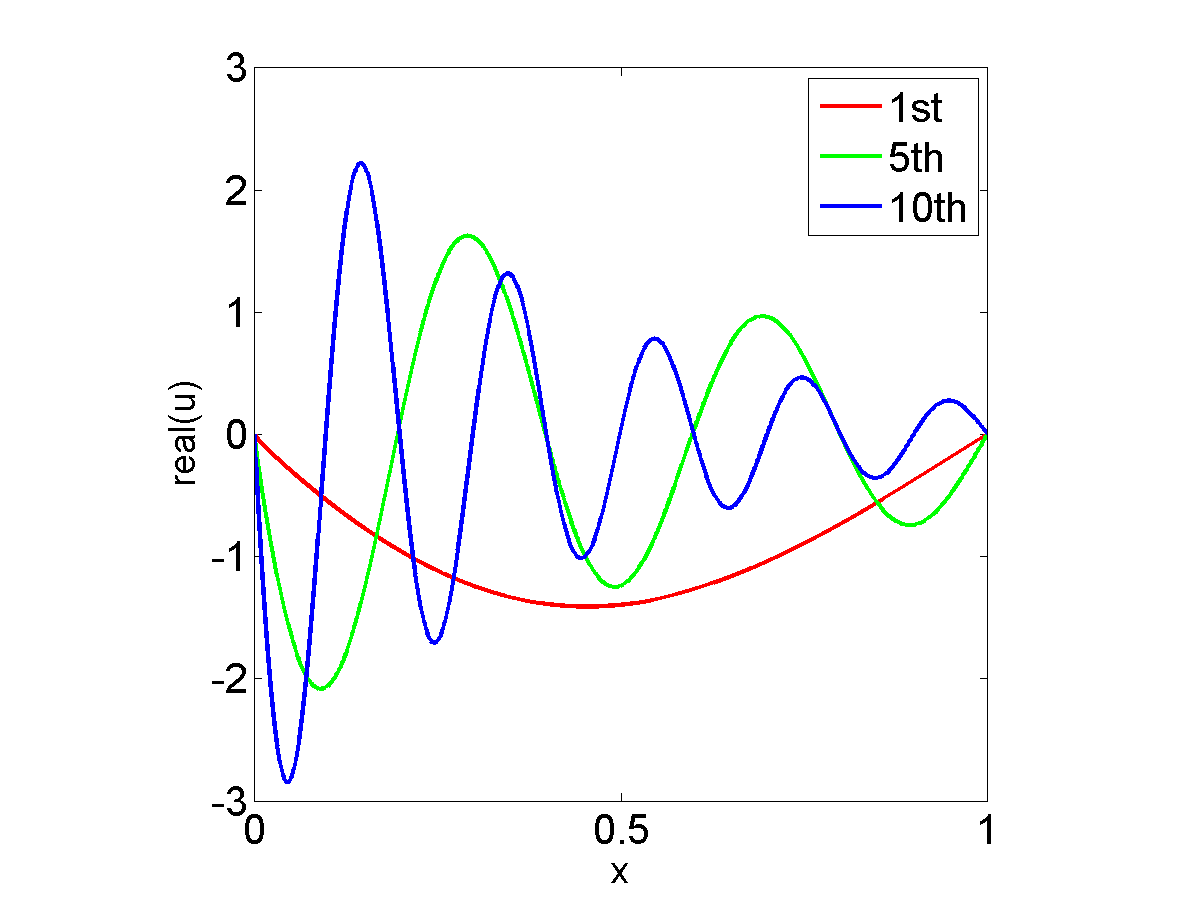} &
    \includegraphics[trim = 1cm 0cm 2cm 0cm, clip=true,width=6cm]{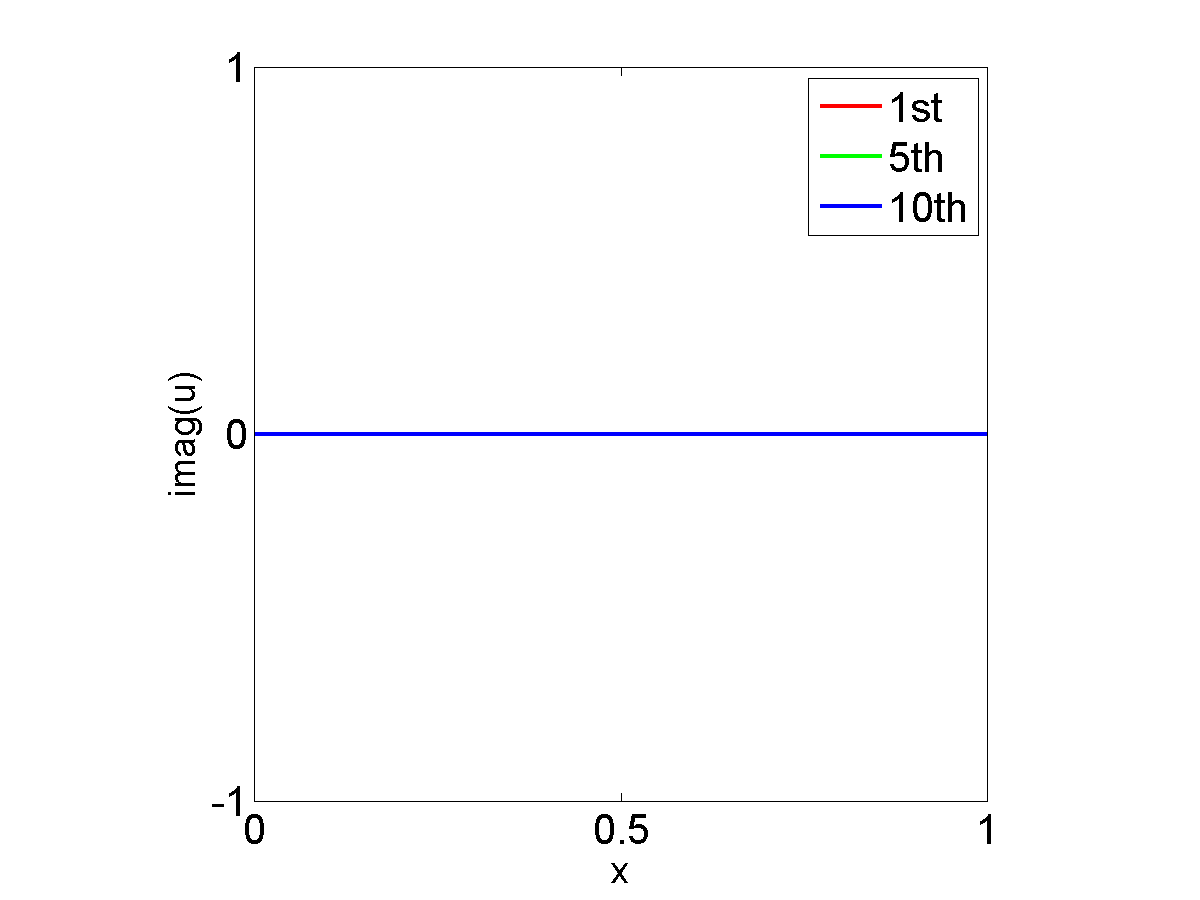}\\
     real part & imaginary part
  \end{tabular}
  \caption{The profile of the first, fifth and tenth eigenfunctions in case of $q_3$, $\alpha=4/3$
  (top row), $\alpha=5/3$ (middle row) and $\alpha=19/10$ (bottom row), Riemann-Liouville case.}\label{fig:eigfcn:riem}
\end{figure}

\subsection{Two extensions} In this part, we discuss two possible extensions of the finite element
formulation.

A first natural idea of extension is to pursue other boundary conditions, e.g., Neumann type, Robin
type or mixed type. The derivation of the variational formulation in Section \ref{sec:prelim} requires
highly nontrivial modifications for these variations. As an illustration, we make the following straightforward attempt
for the Riemann-Liouville case with mixed boundary conditions:
\begin{equation}\label{strongp:rlmixed}
  - \DDR0\alpha u + qu = f \ \ \mbox{ in } D, \qquad \DDR0{\alpha-1}u(0)=u(1)=0,
\end{equation}
where $f\in L^2\II$. Let $g={_0I_x^\alpha} f\in\Hdi0\alpha$, then $\DDR0\alpha g=f$. We thus find that
\begin{equation*}
u= -{_0I_x^\alpha} f +({_0I_x^\alpha} f)(1) x^{\alpha-2}
\end{equation*}
is a solution of \eqref{strongp:rlmixed} in the Riemann-Louiville case (for $q=0$) since
it satisfies the boundary conditions and $\DDR0\alpha x^{\alpha-2}=
(c_\alpha)^{\prime\prime}=0$. 
The representation formula is very suggestive. First, the most ``natural'' variational formulation for
\eqref{strongp:rlmixed}, i.e., find $u\in \Hdi1{\alpha/2}$ such that
\begin{equation*}
  a(u,v) = (f,v)\quad \forall v\in \Hdi1{\alpha/2},
\end{equation*}
with a bilinear form $a(u,v) = -(\DDR0{{\alpha/2}}u,\,
\DDR1{{\alpha/2}}v) + (qu,v)$, cannot be stable on the space $\Hdi1{\alpha/2}$, since in general the
solution $u$ does not admit the $\Hdi1{\alpha/2}$ regularity, due to the presence of
the singular term $x^{\alpha-2}$. Hence, alternative formulations, e.g., of Petrov-Galerkin type,
should be sought for. Second, the solution $u$ can have only very limited Sobolev regularity, and a special
treatment at the left end point is necessary. This clearly illustrates the delicacy of properly
treating boundary conditions in fractional differential equations, which has not received due
attention in the literature.

Due to the one-sidedness of the fractional derivative, one naturally
expects that a Neumann-type boundary condition on the left end point ($x=0$) would influence the
problem structure differently from that on the right end point ($x=1$). Indeed, this seems to be
generally the case. However, there is one special case for the Caputo derivative, where the
spectrum cannot even distinguish the boundary conditions. With a vanishing potential $q=0$, the eigenvalues are both
given by the zeros of the Mittag-Leffler function $E_{\alpha,1}(-\lambda)$ for either Neumann
boundary condition $u'(0)=0$ or $u'(1)=0$. To see this, let $u$ and $v$ be solution to the
following initial value problems:
\begin{equation*}
  \begin{aligned}
    -\DDC0\alpha u = \lambda u\quad \mbox{ in } \II,\ \ u(0)=0,\ \ u'(0) = 1,\\
    -\DDC0\alpha v = \lambda v\quad \mbox{ in } \II,\ \ v(0)=1,\ \ v'(0) = 0.
  \end{aligned}
\end{equation*}
Then following the construction in Section \ref{sec:prelim} (see also \cite{KilbasSrivastavaTrujillo:2006}),
the solutions $u$ and $v$ can be respectively represented by
\begin{equation*}
  u(x) = xE_{\alpha,2}(-\lambda x^\alpha)\quad \mbox{and}\quad v(x) = E_{\alpha,1}(-\lambda x^\alpha).
\end{equation*}
Now in order for $u$ to be an eigenvalue to the fractional SLP $-\DDC0\alpha u =\lambda u$, $u(0)=u'(1)=0$,
$\lambda$ must be a zero of $\frac{d}{dx}xE_{\alpha,2}(-\lambda x^\alpha)|_{x=1} = E_{\alpha,1}(-\lambda x)
|_{x=1}=E_{\alpha,1}(-\lambda)$. Similarly, for $v$ to be an eigenvalue to the fractional SLP $-\DDC0\alpha
v =\lambda v$, $v'(0)=v(1)=0$, $\lambda$ must be a zero of $E_{\alpha,1}(-\lambda)$. This shows the desired
assertion. In particular, this observation indicates the potential nonuniqueness issue for the related
inverse Sturm-Liouville problem with a zero potential: given the complete spectrum, one may not even be able to
determine the boundary condition. This is a bit surprising in view of the one-sidedness of the Caputo derivative.

Throughout we have exclusively focused our discussions on the left-sided Riemann-Liouville and Caputo derivatives.
There are several alternative choices of the spatial derivative, depending on the specific applications.
For example, one may also consider a mixed derivative $\mathbf{D}_\theta^\alpha$ defined by
\begin{equation*}
  \mathbf{D}_\theta^\alpha u = \theta\, \DDR0\alpha u + (1-\theta)\, \DDR1\alpha u\quad \mbox{or}\quad
  \mathbf{D}_\theta^\alpha u = \theta \DDC0\alpha u + (1-\theta)\DDC1\alpha u,
\end{equation*}
where $\theta\in [0,1]$ is a weight. The mixed derivative has been very popular in the mathematical modeling of
spatial fractional diffusion. However, for such mixed derivative, there seems no known variational formulation,
solution representation formula and the regularity pickup. Formally, for the fractional SLP with a mixed
Riemann-Liouville derivative and a zero Dirichlet boundary condition, one would naturally expect that the
respective weak formulation reads: find $u\in \Hd{\alpha/2}\II$ and $\lambda\in \mathbb{C}$ such that
\begin{equation*}
   \theta(\DDR0{\alpha/2} u, \, \DDR1{\alpha/2} v) + (1-\theta)(\DDR1{\alpha/2} u,\, \DDR0{\alpha/2}v) = \lambda(u,v)\quad \forall v\in \Hd{\alpha/2}\II.
\end{equation*}
However, it is still unclear whether this does represent the proper variational formulation, due to a
lack of the solution regularity, especially around the end points. Our numerical experiments
with the variational formulation indicate that the eigenfunctions have singularity only at one end
point, depending on the value of the weight $\theta$: for $\theta>1/2$, the singularity is at the left
end point, whereas for $\theta<1/2$, it is at the right end point. Due to the presence of fractional
derivatives from both end points, the presence of only one single singularity seems counterintuitive.
Nonetheless, in view of the empirically observed solution regularity, the numerical experiments do
confirm a posteriori that the variational formulation in the Riemann-Liouville case seems plausible. However,
a complete mathematical justification of the formulation is still missing. In contrast,
the case of a mixed Caputo derivative is completely unclear.

\subsection{Fractional SLP with a Riemann-Liouville derivative}

In this part, we present a preliminary numerical study of the fractional SLP with a Riemann-Liouville
derivative using the finite element method, since the Caputo case has been studied earlier in
\cite{JinRundell:2012}.

In Section 4.2, we have observed that in the Riemann-Liouville case, there is at least one real eigenvalue,
irrespective of the value of the order $\alpha$.
Then new real eigenvalues always emerge in pairs, and thus the number of real eigenvalues is always odd.
The existence of one real eigenvalue for the case $q=0$ can be rigorously established by appealing to the
Krein-Rutman theorem \cite[Theorem 19.2]{Deimling:1985}; see Appendix \ref{app:existence-real} for details.
In contrast, in the Caputo case, a real eigenvalue exists only if the fractional order $\alpha$ is
sufficiently large (with the critical value lying between $1.59$ and $1.60$, cf. \cite{JinRundell:2012}).
Generally, the structure of eigenvalues and eigenfunctions (for both fractional derivatives) is fairly
elusive. For example, it is not known that the real eigenvalues always appear before complex conjugate
pairs show up, and that the number of real eigenvalues is nondecreasing as the fractional order $\alpha$
increases, albeit both are observed in our numerical experiments.

One naturally wonders how the smallest eigenvalue $\lambda_1(\alpha)$ would vary with the fractional order
$\alpha$. It is tempting to conjecture that the real eigenvalue $\lambda_1(\alpha)$ might be monotonically
increasing in $\alpha$, in view of the observation that asymptotically, the magnitude of the eigenvalues
grows like $(2n\pi)^\alpha$. This is however only partially correct; see Fig. \ref{fig:riem:first}(a) for an illustration.
It is observed that with the increase of the $\alpha$ value, the eigenvalue $\lambda_1(\alpha)$ actually
first monotonically decreases for $\alpha$ up to $1.27$, and then it is monotonically increasing.

\begin{figure}[htp!]
  \centering
  \begin{tabular}{cc}
  \includegraphics[trim = 4cm 0cm 4cm 0cm, clip=true,width=6cm]{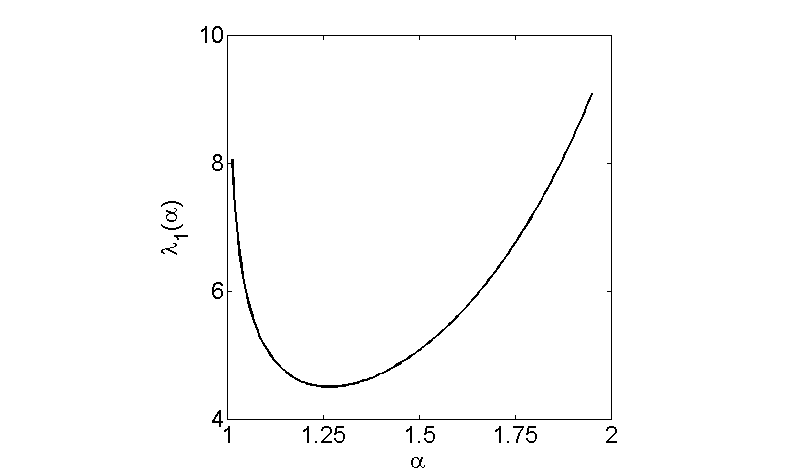}&
  \includegraphics[trim = 4cm 0cm 4cm 0cm, clip=true,width=6cm]{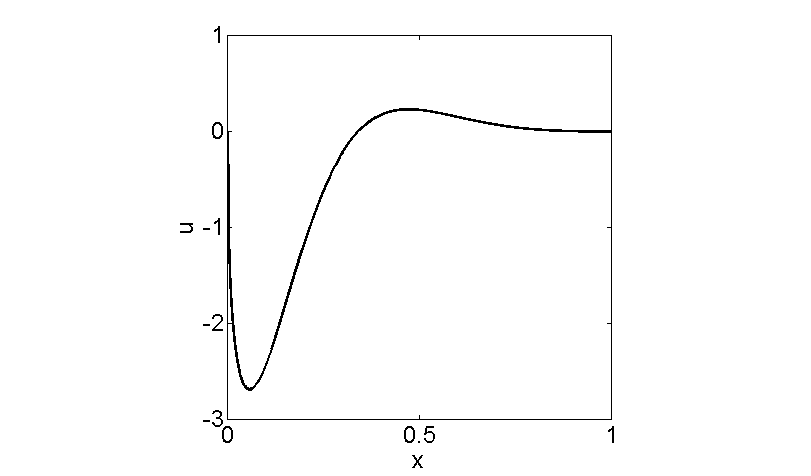}\\
  (a) & (b)
  \end{tabular}
  \caption{(a) The variation of the first eigenvalue $\lambda_1$ with $\alpha$, $q=0$, Riemann-Liouville case.
  (b) The second eigenfunction for $\alpha=1.3396$, $q=0$, Riemann-Liouville case.}\label{fig:riem:first}
\end{figure}

In case of a zero potential, the second and third real eigenvalues appear when the fractional order
$\alpha$ increases from $1.3395$ to $1.3396$, where the complex conjugate pair $19.379372\pm0.170620
\mathrm{i}$ splits into two real eigenvalues $19.283648$ and $19.482320$. Further refinement indicates
these two real eigenvalues are genuinely simple, and with the second eigenfunction has one interior
zero, and the third eigenfunction has two interior zeros, with the second zero located at 0.993, i.e.,
they are linearly independent. The second eigenfunction is shown in Fig. \ref{fig:riem:first}(b), and
the third one is graphically indistinguishable  from the second one. Hence for all practical purposes,
the third eigenfunction does not provide any new information relative to the second one. Further, these two
eigenfunctions are fairly close to each other, and thus the usual interval condition is not valid: actually
the third interior zero of the real part of the next (complex) eigenfunction is located at $0.791$, which
is much smaller than $0.993$. We note that as the order $\alpha$ increases from $1.3395$ to $1.3396$, the
rest of eigenvalues remains fairly stable. It is unclear whether the eigenvalue(s) at bifurcation point
is geometrically/algebraically simple. Naturally, the bifurcation is not stable under the perturbation
of a potential term. For example, in the presence of the potential $q_2$ and at $\alpha=1.3395$, the complex
conjugate pair splits into  $19.085265$ and $20.791554$, and the eigenfunctions are noticeably different.
Such bifurcation behavior can affect greatly the convergence of relevant numerical schemes. For example, in
the frozen Newton method for the inverse Sturm-Liouville problem at the bifurcation point, an inadvertent
choice of the frozen Jacobian at $q=0$ can lead to nonconvergence.

In summary, the behavior of the fractional SLP is fairly intricate. Hence, not surprisingly, it is
very difficult to obtain analytical results. The finite element method developed in this paper provides
an invaluable tool for numerically investigating various ``conjectures'' on the analytical properties.
\section{Concluding remarks}
We have developed a finite element method for fractional Sturm-Liouville problems involving either
the Caputo or Riemann-Liouville derivatives. It is based on novel variational formulations
for fractional differential operators, and rigorous (but suboptimal) error bounds are provided for
the approximate eigenvalues. Numerically, it is observed that the method converges at a second-order
rate for both fractional derivatives, and can provide accurate estimates of multiple eigenvalues in
the presence of either a smooth or nonsmooth potential term. Further, some properties
of the eigenvalues and eigenfunctions are numerically studied.

This work represents only a first step towards rigorous numerics for fractional Sturm-Liouville
problems. There are many possible extensions of the proposed method. First, the case of mixed
left-sided and right-sided Caputo/Riemann-Liouville fractional derivatives occurs often in
practice. However, the proper variational formulation and solution theory, especially regularity
pickup, for such models are still unclear. Second, this work is exclusively concerned with Dirichlet
eigenvalues. It is natural to pursue other boundary conditions, e.g., Neumann or Robin type boundary
conditions. Third, a complete theoretical justification of the superior empirical performances of
the finite element method is of immense interest.

\section*{Acknowledgments}
The research of B. Jin and W. Rundell has been supported by NSF Grant DMS-1319052, R. Lazarov
was supported in parts by NSF Grant DMS-1016525, and J. Pasciak has been supported by NSF Grant
DMS-1216551. The work of all authors has been  supported in parts also  by Award No.
KUS-C1-016-04, made by King Abdullah University of Science and Technology (KAUST).

\appendix
\section{Existence of a real eigenvalue eigenvalue}\label{app:existence-real}

In this appendix, we show that the lowest Dirichlet eigenvalue in the case of a Riemann-Liouville
derivative, with a zero potential, is always positive. To this end, we consider
the solution operator $T: C_0\II\to C_0\II$, $f\to Tf$, with $Tf$ defined by
\begin{equation*}
  Tf = ({_0I_x^\alpha}f)(1)x^{\alpha-1} - {_0I_x^\alpha}f(x).
\end{equation*}
By Theorem \ref{thm:regrl}, the operator $T: C_0\II\to C_0\II$ is compact. Let $K$ be the
set of nonnegative functions in $C_0\II$. Next we show that the operator $T$ is positive on $K$.
Let $f\in C_0\II$, and $f\geq0$. Then
\begin{equation*}
  \begin{aligned}
    Tf(x) =& \frac{1}{\Gamma(\alpha)}\int_0^1(1-t)^{\alpha-1}f(t)dtx^{\alpha-1} - \frac{1}{\Gamma(\alpha)}\int_0^x (x-t)^{\alpha-1}f(t)dt\\
       = & \frac{1}{\Gamma(\alpha)}\int_x^1 (1-t)^{\alpha-1}f(t)dtx^{\alpha-1} + \frac{1}{\Gamma(\alpha)}\int_0^x ((x-xt)^{\alpha-1}-(x-t)^{\alpha-1})f(t)dt.
  \end{aligned}
\end{equation*}
Clearly, for any $x\in D$, the first integral is nonnegative. Similarly, $(x-xt)^{\alpha-1}>
(x-t)^{\alpha-1}$ holds for all $t\in (0,x)$, and thus the second integral is also nonnegative.
Hence, $Tf\in K$, i.e., the operator $T$ is positive. Now it follows directly from the Krein-Rutman theorem
\cite[Theorem 19.2]{Deimling:1985} that the spectral radius of $T$ is an eigenvalue of $T$, and
an eigenfunction $u\in K\setminus\{0\}$.

\bibliographystyle{abbrv}
\bibliography{frac}

\end{document}